\documentclass[11pt]{amsart}
\usepackage{amsmath}
\usepackage{amssymb}
\usepackage{amsthm}
\usepackage{latexsym}
\usepackage{hyperref}
\usepackage{enumerate}
\usepackage[all]{xy}

\newtheorem{thm}{Theorem}[section]
\newtheorem{cor}[thm]{Corollary}
\newtheorem{lem}[thm]{Lemma}
\newtheorem{prop}[thm]{Proposition}

\newcommand{\enuma}[1]{\begin{enumerate}[\textup{(}a\textup{)}] {#1} \end{enumerate}}

\newcommand{\bW}{{\mathbf W}}
\newcommand{\bI}{{\mathbf I}}

\newcommand{\mr}{\mathrm}
\newcommand{\mc}{\mathcal}
\newcommand{\mf}{\mathfrak}

\newcommand{\N}{\mathbb N}
\newcommand{\Z}{\mathbb Z}
\newcommand{\Q}{\mathbb Q}
\newcommand{\R}{\mathbb R}
\newcommand{\C}{\mathbb C}

\newcommand{\vareps}{\epsilon}

\newcommand{\ochi}{\omega \otimes \chi}

\newcommand{\matje}[4]{\left(\begin{smallmatrix} #1 & #2 \\ 
#3 & #4 \end{smallmatrix}\right)}

\newcommand{\cent}{C}

\def\Hom{{\rm Hom}}

\def\ess{{\rm ess}}
\def\Irr{{\rm Irr}}
\def\Gal{{\rm Gal}}
\def\GL{{\rm GL}}
\def\Mat{{\rm M}}
\def\tr{{\rm tr}}
\def\PGL{{\rm PGL}}
\def\SL{{\rm SL}}
\def\JL{{\rm JL}}
\def\St{{\rm St}}

\def\rec{{\rm rec}}
\def\der{{\rm der}}
\def\op{{\rm op}}
\def\char{{\rm char}}

\def\Mod{{\rm Mod}}

\def\swan{{{\rm swan}}}

\begin{document}

\title{The local Langlands correspondence \\for inner forms of $\SL_n$}

\author[A.-M. Aubert]{Anne-Marie Aubert}
\address{Institut de Math\'ematiques de Jussieu -- Paris Rive Gauche, 
U.M.R. 7586 du C.N.R.S., U.P.M.C., 4 place Jussieu 75005 Paris, France}
\email{aubert@math.jussieu.fr}
\author[P. Baum]{Paul Baum}
\address{Mathematics Department, Pennsylvania State University,  University Park, PA 16802, USA}
\email{baum@math.psu.edu}
\author[R. Plymen]{Roger Plymen}
\address{School of Mathematics, Southampton University, Southampton SO17 1BJ,  England \emph{and} 
School of Mathematics, Manchester University, Manchester M13 9PL, England}
\email{r.j.plymen@soton.ac.uk \quad plymen@manchester.ac.uk}
\author[M. Solleveld]{Maarten Solleveld}
\address{Radboud Universiteit Nijmegen, Heyendaalseweg 135, 6525AJ Nijmegen, the Netherlands}
\email{m.solleveld@science.ru.nl}
\date{\today}
\thanks{The second author was partially supported by NSF grant DMS-1200475.}
\maketitle

\begin{abstract}   
Let $F$ be a non-archimedean local field. We establish the local Langlands correspondence 
for all inner forms of the group $\SL_n (F)$. It takes the form of a bijection 
between, on the 
one hand, conjugacy classes of Langlands parameters for $\SL_n (F)$ enhanced with an irreducible 
representation of an S-group and, on the other hand, the union of the spaces of irreducible 
admissible representations of all inner forms of $\SL_n (F)$ up to equivalence.
An analogous result is shown in the archimedean case.

To settle the case where $F$ has positive characteristic, we employ the method of close fields.
We prove that this method is compatible with the local Langlands correspondence for inner
forms of $\GL_n (F)$, when the fields are close enough compared to the depth of the representations.
\end{abstract}

\tableofcontents

\section{Introduction}

Let $F$ be a local field and let $D$ be a division algebra with centre $F$, of
dimension $d^2$ over $F$. Then $G = \GL_m (D)$ is an inner form of $\GL_{md}(F)$
and it is endowed with a reduced norm map Nrd$: \GL_m (D) \to F^\times$.
The group $G^\sharp := \ker (\text{Nrd} : G \to F^\times)$ is an inner form of
$\SL_n (F)$. In this paper we will complete the local Langlands correspondence
for $G^\sharp$. 

We sketch how it goes and which part of it is new. Let $\Irr (H)$ denote the set 
of (isomorphism classes of) irreducible admissible $H$-representations. If $H$
is a reductive group over a local field, we denote the collection of
equivalence classes of Langlands parameters for $H$ by $\Phi (H)$. The local Langlands
correspondence (LLC) for $\GL_n (F)$ was established in the important papers
\cite{Lan,LRS,HaTa,Hen2,Zel}. Together with the Jacquet--Langlands correspondence
this provides the LLC for inner forms $G = \GL_m (D)$ of $\GL_n (F)$, see 
\cite{HiSa,ABPS1}. For these groups every L-packet $\Pi_\phi (G)$ is a singleton 
and the LLC is a canonical bijective map
\begin{equation}\label{eq:1.2}
\mathrm{rec}_{D,m} : \Irr (\GL_m (D)) \to \Phi (\GL_m (D)) . 
\end{equation}
The LLC for inner forms of $\SL_n (F)$ is derived from the above, in
the sense that every L-packet for $G^\sharp$ consists of the irreducible
constituents of $\mathrm{Res}_{G^\sharp}^G (\Pi_\phi (G))$. Of course these
L-packets have more than one element in general. To parametrize the members of
$\Pi_\phi (G^\sharp)$ one must enhance the Langlands parameter $\phi$ with an
irreducible representation of a suitable component group. 
This idea originated for unipotent representations of $p$-adic reductive groups 
in \cite[\S~1.5]{LuSquare}. For $\SL_n (F) ,\; \phi$
is a map from the Weil--Deligne group of $F$ to $\PGL_n (\C)$ and a correct choice
is the group of components of the centralizer of $\phi$ in $\PGL_n (\C)$, see
\cite{GeKn}. In general a more subtle component group $\mc S_\phi$ is needed, see 
\cite{Vog,Art3}. 

Let $\Phi^e (G^\sharp)$ be the collection of equivalence classes 
$(\phi,\rho)$ of a Langlands parameter $\phi$ for $G^\sharp$, enhanced with
$\rho \in \Irr (\mc S_\phi)$. The LLC for $G^\sharp$ should be an injective map
\begin{equation}\label{eq:1.1}
\Irr (G^\sharp) \to \Phi^e (G^\sharp) , 
\end{equation}
which satisfies several naturality properties. The map will almost never be 
surjective, but for every $\phi$ which is relevant for $G^\sharp$ the image
should contain at least one pair $(\phi,\rho)$. This form of the LLC was proven
for generic representations of $G^\sharp$ in \cite{HiSa}, under the assumption
that the underlying local field has characteristic zero.

A remarkable aspect of Langlands' conjectures \cite{Vog} is that it is better 
to consider not just one reductive group at a time, but all inner forms of a 
given group simultaneously. Inner forms share the same
Langlands dual group, so in \eqref{eq:1.1} the right hand side is the same
for all inner forms $H$ of the given group. The hope is that one can turn 
\eqref{eq:1.1} into a bijection by defining a suitable equivalence relation on
the set of inner forms and taking the corresponding union of the sets $\Irr (H)$ 
on the left hand side. Such a statement was proven for unipotent representations
of simple $p$-adic groups in \cite{Lus}.

Let us make this explicit for inner forms of $\GL_n (F)$, respectively $\SL_n (F)$.
We define the equivalence classes of such inner forms to be in bijection with the
isomorphism classes of central simple $F$-algebras of dimension $n^2$ via 
$\Mat_m (D) \mapsto \GL_m (D)$, respectively $\Mat_m (D) \mapsto \GL_m (D)_{\der}$.

As Langlands dual group we take $\GL_n (\C)$, respectively $\PGL_n (\C)$. To deal
with inner forms it is advantageous to consider the conjugation action of $\SL_n (\C)$ 
on these two groups. It induces a natural action of $\SL_n (\C)$
on the collection of Langlands parameters for $\GL_n (F)$ or $\SL_n (F)$.
For any such parameter $\phi$ we can define
\begin{equation}\label{eq:1.3}
\begin{aligned}
& C(\phi) = Z_{\SL_n (\C)}(\text{im } \phi) , \\
& \mc S_\phi = C(\phi) / C(\phi)^\circ , \\
& \mc Z_\phi = Z (\SL_n (\C)) / Z (\SL_n (\C)) \cap C(\phi)^\circ \cong
Z (\SL_n (\C)) C(\phi)^\circ / C(\phi)^\circ .
\end{aligned}
\end{equation}
Notice that the centralizers are taken in $\SL_n (\C)$ and not in the Langlands dual
group. The groups in \eqref{eq:1.3} are related to the more usual component group
$S_\phi := Z_{\PGL_n (\C)}(\text{im } \phi) / Z_{\PGL_n (\C)}(\text{im } \phi)^\circ$
by the short exact sequence
\[
1 \to \mc Z_\phi \to \mc S_\phi \to S_\phi \to 1.
\]
Hence $\mc S_\phi$ has more irreducible representations than $S_\phi$. Via the Langlands
correspondence the additional ones are associated to irreducible representations of
non-split inner forms of $\GL_n (F)$ or $\SL_n (F)$.

For example, consider a Langlands parameter $\phi$ for $\GL_2 (F)$ which is 
elliptic, that is, whose image is not contained in any torus of $\GL_2 (\C)$. 
Then $S_\phi = 1$
but $\mc S_\phi = Z(\SL_2 (\C)) \cong \{\pm 1\}$. The pair $(\phi,\mathrm{triv}_{\mc S_\phi})$
parametrizes an essentially square-integrable representation of $\GL_2 (F)$ and 
$(\phi,\mathrm{sgn}_{\mc S_\phi})$ parametrizes an irreducible representation of the inner 
form $D^\times$, where $D$ denotes a noncommutative division algebra of dimension $4$ over $F$.

For general linear groups over local fields we prove a result which was already known 
to experts, but which we could not find in the literature:

\begin{thm}\label{thm:1.1} (see Theorem \ref{thm:5.2}) \\
There is a canonical bijection between:
\begin{itemize}
\item pairs $(G,\pi)$ with $\pi \in \Irr (G)$ and $G$ an inner form of $\GL_n (F)$,
considered up to equivalence;
\item $\GL_n (\C)$-conjugacy classes of pairs $(\phi,\rho)$ with $\phi \in \Phi (\GL_n (F))$
and $\rho \in \Irr (\mc S_\phi)$. 
\end{itemize}
\end{thm}

For these Langlands parameters $\mc S_\phi = \mc Z_\phi$ and a character of $\mc Z_\phi$ 
determines an
inner form of $\GL_n (F)$ via the Kottwitz isomorphism \cite{Kot}. In contrast with the
usual LLC, our packets for general linear groups need not be singletons. To be precise,
the packet $\Pi_\phi$ contains the unique representation $\mathrm{rec}_{D,m}^{-1}(\phi)$  
of $G = \GL_m (D)$ if $\phi$ is relevant for $G$, and no $G$-representations otherwise.

A similar result holds for special linear groups, but with a few modifications. Firstly,
one loses canonicity, because in general there seems to be no natural way to parametrize
the members of an L-packet $\Pi_\phi (G^\sharp)$ (if there are more than one). Secondly,
the quaternion algebra $\mathbb H$ turns out to occupy an exceptional position. Our 
local Langlands correspondence for inner forms of the special linear group over a local
field $F$ can be stated as follows:

\begin{thm}\label{thm:1.2} (see Theorems \ref{thm:6.5} and \ref{thm:6.6}) \\
There exists a correspondence between:
\begin{itemize}
\item pairs $(G^\sharp,\pi)$ with $\pi \in \Irr (G^\sharp)$ and $G^\sharp$ an inner form 
of $\SL_n (F)$, considered up to equivalence;
\item $\SL_n (\C)$-conjugacy classes of pairs $(\phi,\rho)$ with $\phi \in \Phi (\SL_n (F))$
and $\rho \in \Irr (\mc S_\phi)$,
\end{itemize}
which is almost bijective, the only exception being that pairs 
$(\SL_{n/2}(\mathbb H),\pi)$ correspond to two parameters $(\phi,\rho_1)$ and $(\phi,\rho_2)$.
\enuma{
\item The group $G^\sharp$ determines $\rho \big|_{\mc Z_\phi}$ and conversely.
\item The correspondence satisfies the desired properties from \cite[\S 10.3]{Bor}, with
respect to restriction from inner forms of $\GL_n (F)$, temperedness and essential 
square-integrability of representations.
}
\end{thm}

In the archimedean case the classification of $\Irr (\SL_m (D))$ is well-known, at least
for $D \neq \mathbb H$. The main value of our result lies in the strong analogy with
the non-archimedean case. The reason for the lack of bijectivity for the special linear
groups over the quaternions is easily identified. Namely, the reduced norm map for
$\mathbb H$ satisfies Nrd$(\mathbb H^\times) = \R_{>0}$ whereas for all other local
division algebras $D$ with centre $F$ the reduced norm map is surjective, that is,
Nrd$(D^\times) = F^\times$. Of course there are various ad hoc ways to restore the 
bijectivity in Theorem \ref{thm:1.2}, for example by decreeing that $SL_m (\mathbb H)$
appears twice among the equivalence classes of inner forms of $SL_{2m}(\R)$. This can
be achieved in natural way with strong inner forms, as in \cite{Ada}.
But one may also argue that for $\SL_m (\mathbb H)$
one would actually be better off without any component groups.

For $p$-adic fields $F$, the above theorem can be derived rather quickly from the work
of Hiraga and Saito \cite{HiSa}. 

By far the most difficult case of Theorem \ref{thm:1.2} is that where the local field
$F$ has positive characteristic. The paper \cite{HiSa} does not apply in this case,
and it seems hard to generalize the techniques from \cite{HiSa} to fields of
positive characteristic.

Our solution is to use the method of close fields to reduce it to the $p$-adic case.
Let $F$ be a local field of characteristic $p ,\; \mf o_F$ its ring of integers and
$\mf p_F$ the maximal ideal of $\mf o_F$. There exist finite extensions $\widetilde F$
of $\mathbb Q_p$ which are $l$-close to $F$, which means that $\mf o_F / \mf p_F^l$ 
is isomorphic to the corresponding ring for $\widetilde{F}$. Let $\widetilde D$
be a division algebra with centre $\widetilde F$, such that $D$ and $\widetilde D$
have the same Hasse invariant. Let $K_r$ be the standard congruence subgroup of
level $r \in \N$ in $\GL_m (\mf o_D)$ and let $\Irr (G,K_r)$ be the set of irreducible
representations of $G = \GL_m (D)$ with nonzero $K_r$-invariant vectors. Define
$\widetilde{K_r} \subset \GL_m (\widetilde{D})$ and 
$\Irr (\GL_m (\widetilde{D}),\widetilde{K_r})$ in the same way.

For $l$ sufficiently large compared to $r$, the method of close fields provides
a bijection
\begin{equation}\label{eq:1.4}
\Irr (\GL_m (D),K_r) \to \Irr (\GL_m (\widetilde{D}),\widetilde{K_r})
\end{equation}
which preserves almost all the available structure \cite{Bad1}. But this is not 
enough for Theorem \ref{thm:1.2}, we also need to relate to the local Langlands
correspondence. The $l$-closeness of $F$ and $\widetilde{F}$ implies
that the quotient of the Weil group of $F$ by its $l$-th ramification subgroup
is isomorphic to the analogous object for $\widetilde F$ \cite{Del}. This yields
a natural bijection
\begin{equation}\label{eq:1.5}
\Phi_l (\GL_m (D)) \to \Phi_l (\GL_m (\widetilde{D})) 
\end{equation}
between Langlands parameters that are trivial on the respective $l$-th ramification
groups. We show that:

\begin{thm}\label{thm:1.3} (see Theorems \ref{thm:8.1} and \ref{thm:8.2}) \\
Suppose that $F$ and $\widetilde F$ are $l$-close and that $l$ is sufficiently large 
compared to $r$. Then the maps \eqref{eq:1.2}, \eqref{eq:1.4}
and \eqref{eq:1.5} form a commutative diagram 
\[
\begin{array}{ccc}
\Irr (\GL_m (D),K_r) & \rightarrow & 
\Irr (\GL_m (\widetilde{D}), \widetilde{K}_r) \\
\downarrow & & \downarrow \\
\Phi_l (\GL_m (D)) & \rightarrow & \Phi_l (\GL_m (\widetilde{D})) .
\end{array}
\]
In the special case $D = F$ and $\widetilde D = \widetilde F$ 
this holds for all $l > 2^{n-1} r$.
\end{thm}

The special case was also proven by Ganapathy \cite{Gan1,Gan2}, but without
an explicit lower bound on $l$.

Theorem \ref{thm:1.3} says that the method of close fields essentially preserves Langlands 
parameters. The proof runs via the only accessible characterization of the LLC for general
linear groups: by means of $\vareps$- and $\gamma$-factors of pairs of representations
\cite{Hen1}. 

To apply Henniart's characterization with maximal effect, we establish a result with
independent value. Given a Langlands parameter $\phi$, we let $d(\phi)$ be the smallest
rational number such that $\phi \notin \Phi_{d(\phi)}(\GL_n (F))$, that is, the smallest 
number such that $\phi$ is nontrivial on the $d(\phi)$-th ramification group of the Weil 
group of $F$. For a supercuspidal representation $\pi$ of $\GL_n (F)$, let $d (\pi)$ be its
normalized level, as in \cite{Bus}. 

\begin{prop}\label{prop:1.4} (see Proposition \ref{lem:pi}) \\
The local Langlands correspondence for supercuspidal representations of $\GL_n (F)$
preserves depths, in the sense that 
\[
d(\pi) = d (\rec_{F,n}(\pi)) . 
\]
\end{prop}

\textbf{Acknowledgements.}
The authors wish to thank Ioan Badulescu for interesting emails about the method
of close fields, Wee Teck Gan for explaining some subtleties of inner forms and
Guy Henniart for pointing out a weak spot in an earlier version of Theorem \ref{thm:LLCr}.

\section{Inner forms of $\GL_n (F)$}
\label{sec:GLn}

Let $F$ be a local field and let $D$ be a division algebra with centre $F$, of dimension 
$\dim_F (D) = d^2$. The $F$-group $\GL_m (D)$ is an inner form of $\GL_{md}(F)$, and 
conversely every inner form of $\GL_n (F)$ is isomorphic to such a group.

In the archimedean case there are only three possible division algebras:
$\R, \C$ and $\mathbb H$. The group $\GL_m (\mathbb H)$ is an inner form of $\GL_{2m}(\R)$, 
and (up to isomorphism) that already accounts for all the inner forms of the groups 
$\GL_n (\R)$ and $\GL_n (\C)$. One can parametrize these inner forms with characters of order 
at most two of $Z (\SL_n(\C))$, such that $\GL_n (F)$ is associated to the trivial character and
\begin{equation}\label{eq:5.7} 
\GL_m (\mathbb H) \text{ corresponds to the character of order two of } Z (\SL_{2m}(\C)).
\end{equation}
Until further notice we assume that $F$ is non-archimedean.
Let us make our equivalence relation on the set of inner forms of $\GL_n (F)$ explicit.
We start with the Galois cohomology group $H^1 (F,\PGL_n (F))$. It parametrizes the 
isomorphism classes of central simple $F$-algebras of dimension $n^2$.
By \cite[Proposition 6.4]{Kot} there exists a natural bijection
\begin{equation}\label{eq:Kott}
H^1 (F,\PGL_n (F)) \to \Irr \big( Z(\SL_n (\C)) \big) = 
\Irr \big( \{z \in \C^\times : z^n = 1\} \big) .
\end{equation}
Clearly the map
\[
\Irr (\{z \in \C^\times : z^n = 1\}) \to \{z \in \C^\times : z^n = 1\} : 
\chi \mapsto \chi (\exp (2\pi \sqrt{-1} / n))
\]
is bijective. The composition of these two maps can be also be interpreted in terms
of classical number theory. For $\Mat_m (D)$ with $md = n$, the Hasse-invariant
$h(D)$ (in the sense of Brauer theory) is a primitive $d$-th root of unity. The 
element of $H^1 (F,\PGL_n (F))$ associated to $\Mat_m (D)$ has the same image
$h(D)$ in $\{z \in \C^\times : z^n = 1\}$.
In particular $1 \in \C^\times$ is associated to $\Mat_n (F)$ and the primitive 
$n$-th roots of unity correspond to division algebras of dimension $n^2$ over 
their centre $F$.

We use the same equivalence relation on the collection of inner forms of $\GL_n (F)$.
In other words, we define that the equivalence classes of such inner forms are in 
bijection with the isomorphism classes of central simple $F$-algebras of dimension
$n^2$ via $\Mat_m (D) \mapsto \GL_m (D)$. We warn the reader that this is not the same
as isomorphism classes of inner forms of $\GL_n (F)$. Namely, if $h(D') = h(D)^{-1}$,
then $\Mat_m (D')$ is isomorphic to the opposite algebra of $\Mat_m (D)$, and
\[
\GL_m (D) \to \GL_m (D^{\op}) \cong \GL_m (D') : x \mapsto x^{-1}  
\]
is  a group isomorphism. All isomorphisms between the groups $\GL_m (D)$ arise
in this way.

Furthermore there is a standard presentation of the division algebras $D$. Let $L$ 
be the unique unramified extension of $F$ of degree $d$ and let $\chi$ be the character
of Gal$(L/F) \cong \Z / d \Z$ which sends the Frobenius automorphism to $h(D)$.
If $\varpi_F$ is a uniformizer of $F$, then $D$ is isomorphic to the cyclic algebra 
$[L/F,\chi,\varpi_F]$, see Definition IX.4.6 and Corollary XII.2.3 of \cite{Wei2}. 
We will call a group of the form
\begin{equation}\label{eq:5.6}
\GL_m \big( [L/F,\chi,\varpi_F] \big)
\end{equation}
a standard inner form of $\GL_n (F)$.

The local Langlands correspondence for $G = \GL_m (D)$ has been known to experts for
considerable time, although it did not appear in the literature until recently
\cite{HiSa,ABPS1}. We need to understand it well for our later arguments, so we recall
its construction. It generalizes and relies on the LLC for general linear groups:
\[
\mathrm{rec}_{F,n} : \Irr (\GL_n (F)) \to \Phi (\GL_n (F)) .
\]
The latter was proven for supercuspidal representations in \cite{LRS,HaTa,Hen2}, 
and extended from there to $\Irr (\GL_n (F))$ in \cite{Zel}.

As $G$ is an inner form of $\GL_n (F) ,\; \check G = \GL_n (\C)$ and the action of
Gal$(\overline{F} / F)$ on $\GL_n (\C)$ determined by $G$ is by inner automorphisms.
Therefore we may take as Langlands dual group ${}^{L}G = \check G = \GL_n (\C)$.
Let $\phi \in \Phi (\GL_n (F))$ and let $\check M \subset \GL_n (\C)$ be a Levi subgroup
that contains im$(\phi)$ and is minimal for this property. As for all Levi subgroups,
\[
\check M \cong \GL_{n_1} (\C) \times \cdots \times \GL_{n_k}(\C)
\]
for some integers $n_i$ with $\sum_{i=1}^k n_i = n$. Then $\phi$ is relevant for $G$ if
and only if $\check M$ corresponds to a Levi subgroup $M \subset G$. This is equivalent
to $m_i := n_i / d$ being an integer for all $i$. Moreover in that case 
\begin{equation}\label{eq:5.1}
M \cong \GL_{m_i}(D) \times \cdots \times \GL_{m_k}(D) .
\end{equation}
Now consider any $\phi \in \Phi (G)$. Conjugating by a suitable element of $\check G$, 
we can achieve that
\begin{itemize}
\item $\check M = \prod_{i=1}^l \GL_{n_i}(\C)^{e_i}$ and $M = \prod_{i=1}^l \GL_{m_i}(D)^{e_i}$ 
are standard Levi subgroups of $\GL_n (C)$ and $\GL_m (D)$, respectively;
\item $\phi = \prod_{i=1}^l \phi_i^{\otimes e_i}$ with $\phi_i \in \Phi (\GL_{m_i}(D))$
and im$(\phi_i)$ not contained in any proper Levi subgroup of $\GL_{n_i}(\C)$;
\item $\phi_i$ and $\phi_j$ are not equivalent if $i \neq j$.
\end{itemize}
Then $\mathrm{rec}_{F,n_i}^{-1}(\phi_i) \in \Irr (\GL_{n_i}(F))$ is essentially 
square-integrable. Recall that the Jacquet--Langlands correspondence 
\cite{Rog,DKV,Bad1} is a natural bijection
\[
\JL : \Irr_{\ess L^2} (\GL_m (D)) \to \Irr_{\ess L^2}(\GL_n (F)) 
\]
between essentially square-integrable irreducible representations of
$G = \GL_m (D)$ and $\GL_n (F)$. It gives
\begin{equation}\label{eq:5.5}
\begin{aligned}
& \omega_i := \JL^{-1} \big( \mathrm{rec}_{F,{n_i}}^{-1}(\phi_i) \big) 
\in \Irr_{\ess L^2}(\GL_{m_i}(D)) , \\
& \omega := \prod\nolimits_{i=1}^l \omega_i^{\otimes e_i} \in \Irr_{\ess L^2}(M) .
\end{aligned}
\end{equation}
We remark that $\omega$ is square-integrable modulo centre if and only all 
$\mathrm{rec}_{F,{n_i}}^{-1}(\phi_i)$ are so, because this property is preserved
by the Jacquet--Langlands correspondence. The Zelevinsky classification for
$\Irr (\GL_{n_i}(F))$ \cite{Zel} (which is used for $\mathrm{rec}_{F,n_i}$) shows 
that, in the given circumstances, this is equivalent to $\phi_i$ being bounded. 
Thus $\omega$ is square-integrable modulo centre if and only $\phi$ is bounded.

The assignment $\phi \mapsto (M,\omega)$ sets up a bijection
\begin{equation}\label{eq:LLC2}
\Phi (G) \longleftrightarrow \{ (M,\omega) : M \text{ a Levi subgroup of } G ,
\omega \in \Irr_{\ess L^2}(M) \} / G .
\end{equation}
It is known from \cite[Theorem B.2.d]{DKV} and \cite{Bad2} that for inner forms of
$\GL_n (F)$ normalized parabolic induction sends irreducible square-integrable
(modulo centre) representations to irreducible tempered representations. 
Together with the Langlands classification \cite{Lan,Kon} this implies that there exists 
a natural bijection between $\Irr (G)$ and the right hand side of \eqref{eq:LLC2}. 
It sends $(M,\omega)$ to the unique
Langlands quotient $L(P,\omega)$ of $I_P^G (\omega)$, where $P \subset G$ is a
parabolic subgroup with Levi factor $M$, with respect to which $\omega$ is positive.

The composition 
\begin{equation}\label{eq:5.2}
\Phi (G) \to \Irr (G) : \phi \mapsto (M,\omega) \mapsto L(P,\omega) 
\end{equation}
is the local Langlands correspondence for $\GL_m (D)$. 

By construction $L(P,\omega)$ is essentially square-integrable if and only if $\phi$
is not contained in any proper Levi subgroup of $\GL_m (D)$. By the uniqueness part
of the Langlands classification \cite[Theorem 3.5.ii]{Kon} $L(P,\omega)$ is tempered 
if and only if $\omega$ is square-integrable
modulo centre, which by the above is equivalent to boundedness of $\phi \in \Phi (G)$.

We note that all the R-groups and component groups are trivial for $G$, and that all
the L-packets $\Pi_\phi (G) = \{ L(P,\omega) \}$ are singletons. This means that
\eqref{eq:5.2} is bijective, and that it has an inverse
\begin{equation}\label{eq:5.4}
\mathrm{rec}_{D,m} : \Irr (G) \to \Phi (G) .
\end{equation}
Because both the LLC for $\Irr_{\ess L^2}(\GL_{n_i}(F))$, the Jacquet--Langlands 
correspondence and $I_P^G$ respect tensoring with unramified characters, 
$\mathrm{rec}_{D,m}(L(P,\ochi))$ and $\mathrm{rec}_{D,m}(L(P,\omega))$ differ only
by the unramified Langlands parameter for $M$ which corresponds to $\chi$. 

In the archimedean case Langlands \cite{Lan} himself established the correspondence 
between the irreducible admissible representations of $\GL_m (D)$ and Langlands
parameters. The paper \cite{Lan} applies to all real reductive groups, but it completes 
the classsification only if parabolic induction of tempered representations of Levi
subgroups preserves irreducibility. That is the case for $\GL_n (\C)$ by the Borel--Weil 
theorem, and for $\GL_n (\R)$ and $\GL_m (\mathbb H)$ by \cite[\S 12]{BaRe}.

The above method to go from essentially square-integrable
to irreducible admissible representations is essentially the same over all local fields,
and stems from \cite{Lan}. There also exists a Jacquet--Langlands correspondence over
local archimedean fields \cite[Appendix D]{DKV}. Actually it is very simple, the
only nontrivial cases are $\GL_2 (\R)$ and $\mathbb H$. Therefore it is justified
to say that \eqref{eq:5.5}--\eqref{eq:5.4} hold in the archimedean case.

With the S-groups from \cite{Art3} we can build a more subtle version of \eqref{eq:5.4}.
Since $Z_{\GL_n (\C)}(\phi)$ is connected,
\begin{multline}\label{eq:5.3}
\mc{S}_\phi = \cent (\phi) / \cent (\phi)^\circ = 
Z (\SL_n (\C)) Z_{\SL_n (\C)}(\phi )^\circ / Z_{\SL_n (\C)}(\phi )^\circ \cong \\
Z(\SL_n (\C)) / \big( Z(\SL_n (\C)) \cap Z_{\SL_n (\C)}(\phi )^\circ \big) .
\end{multline}
Let $\chi_G \in \Irr \big( Z(\SL_n (\C)) \big)$ be the character associated 
to $G$ via \eqref{eq:Kott} or \eqref{eq:5.7}.

\begin{lem}\label{lem:5.1}
A Langlands parameter $\phi \in \Phi (\GL_n (F))$ is relevant for $G = \GL_m (D)$
if and only if $\ker \chi_G \supset Z(\SL_n (\C)) \cap \cent (\phi)^\circ$. 
\end{lem}
\begin{proof}
This is \cite[Lemma 9.1]{HiSa} for inner forms of $\GL_n (F)$. Although a standing 
assumption in \cite{HiSa} is that char$(F) = 0$, the proof of this result works just
as well in positive characteristic.
\end{proof}

We regard
\[
\Phi^e (\text{inn } \GL_n (F)) := \{ (\phi,\rho) : \phi \in \Phi (\GL_n (F)),
\rho \in \Irr (\mc{S}_\phi) \}
\]
as the collection of enhanced Langlands parameters for all inner forms of $\GL_n (F)$.
With this set can establish the local Langlands correspondence for all such inner
forms simultaneously. To make it bijective, we must choose one group in each equivalence
class of inner forms of $\GL_n (F)$. In the archimedean case it suffices to say that
we use the quaternions, and in the non-archimedean case we take the standard inner 
forms \eqref{eq:5.6}.

\begin{thm}\label{thm:5.2}
Let $F$ be a local field. There exists a canonical bijection
\[
\begin{array}{lcl} 
\Phi^e (\mathrm{inn} \; \GL_n (F)) & \to & \{ (G,\pi) : G \text{ standard inner form of } 
\GL_n (F), \pi \in \Irr (G) \} , \\
(\phi,\chi_G) & \mapsto & (G,\Pi_\phi (G)) .
\end{array} 
\]
\end{thm}
\begin{proof}
The elements of $\Phi^e (\text{inn } \GL_n (F))$ with a fixed $\phi \in \Phi (\GL_n (F))$ are
\begin{equation} \label{eq:5.8}
\big\{ (\phi,\chi) : \chi \in \Irr \big( Z (\SL_n (\C)) \big), \ker \chi \supset
Z(\SL_n (\C)) \cap \cent (\phi)^\circ \big \} .
\end{equation}
First we consider the non-archimedean case.
By Lemma \ref{lem:5.1} and \eqref{eq:Kott}, \eqref{eq:5.8} is in bijection with the 
equivalence classes of inner forms $G$ of $\GL_n (F)$ for which $\phi$ is relevant.
Now apply the LLC for $G$ \eqref{eq:5.2}.

In the archimedean case the above argument does not suffice, because some characters
of $Z(\SL_n (\C))$ do not parametrize an inner form of $\GL_n (F)$. We proceed by direct
calculation, inspired by \cite[\S 3]{Lan}.

Suppose that $F = \C$. Then $\mathbf{W}_F = \C^\times$ and im$(\phi)$ is just a real
torus in $\GL_n (\C)$. Hence $Z_{\GL_n (\C)}(\phi)$ is a Levi subgroup of $\GL_n (\C)$
and $C(\phi) = Z_{\SL_n (\C)}(\phi)$ is the corresponding Levi subgroup of $\SL_n (\C)$.
All Levi subgroups of $\SL_n (\C)$ are connected, so $\mc S_\phi = C(\phi) / C(\phi)^\circ 
= 1$. Consequently $\Phi^e (\text{inn } \GL_n (\C)) = \Phi (\GL_n (\C))$, and the theorem
for $F = \C$ reduces to the Langlands correspondence for $\GL_n (\C)$.

Now we take $F = \R$. Recall that its Weil group is defined as
\[
\mathbf{W}_\R = \C^\times \cup \C^\times \tau \text{, where } \tau^2 = -1 \text{ and }
\tau z \tau^{-1} = \overline{z} . 
\]
Let $M$ be a Levi subgroup of $\GL_n (\C)$ which contains the image of $\phi$ and is 
minimal for this property. Then $\phi (\C^\times)$ is contained in a unique maximal
torus $T$ of $M$. By replacing $\phi$ by a conjugate Langlands parameter, 
we can achieve that 
\[
M = \prod\nolimits_{i=1}^n \GL_i (\C)^{n_i}
\] 
is standard and that $T$ is the torus of diagonal
matrices. Then the projection of $\phi (\mathbf{W}_\R)$ on each factor $\GL_i (\C)$ of $M$
has a centralizer in $\GL_i (\C)$ which does not contain any torus larger than $Z(\GL_i (\C))$.
On the other hand $\phi (\tau)$ normalizes $Z_M (\phi (\C^\times)) = T$,
so $Z_{GL_n (\C)}(\phi) = Z_T (\phi (\tau))$. It follows that $n_i = 0$ for $i \geq 2$. 

The projection of $\phi (\tau)$ 
on each factor $\GL_2 (\C)$ of $M$ is either $\matje{0}{1}{-1}{0}$ or $\matje{0}{-1}{1}{0}$. 
Hence $Z_{\GL_n (\C)}(\phi)$ contains the torus
\[
T_\phi := (\C^\times )^{n_1} \times Z (\GL_2 (\C))^{n_2} . 
\]
Suppose that $n_1 > 0$. Then the intersection $T_\phi \cap \SL_n (\C)$ is connected, so 
\[
Z(\SL_n (\C)) / \big( Z(\SL_n (\C)) \cap Z_{\SL_n (\C)}(\phi )^\circ \big) = 1 .
\]
Together with \eqref{eq:5.3} this shows that $\mc S_\phi = 1$ if $n$ is odd or if $n$ is
even and $\phi$ is not relevant for $\GL_{n/2}(\mathbb H)$.

Now suppose $n_1 = 0$. Then $n = 2 n_2 ,\; \phi$ is relevant for $\GL_{n_2}(\mathbb H)$ and
$T_\phi = \big\{ (z_j I_2)_{j=1}^{n_2} : z_j \in \C^\times \big\}$. We see that $T_\phi \cap
\SL_n (\C)$ has two components, determined by whether $\prod_{j=1}^{n_2} z_j$ equals 1 or -1.
Write $\phi = \prod_{j=1}^{n_2} \phi_j$ with $\phi_j \in \Phi (\GL_2 (\R))$. We may assume
that $\phi$ is normalized such that, whenever $\phi_j$ is $\GL_2 (\C)$-conjugate to $\phi_{j'}$,
actually $\phi_{j'} = \phi_j = \phi_k$ for all $k$ between $j$ and $j'$. 
Then $Z_{\Mat_n (\C)} (\phi)$ is isomorphic to a standard Levi subalgebra $A$ 
of $\Mat_{n_2}(\C)$,
via the ring homomorphism
\[
\Mat_{n_2}(\C) \to \Mat_n (\C) = \Mat_{n_2}(\Mat_2 (\C)) \text{ induced 
by } z \mapsto z I_2 . 
\]
Hence $Z_{\SL_n (\C)}(\phi) \cong \{ a \in A : \det (a)^2 = 1 \}$, which clearly has two
components. This shows that $|\mc S_\phi | = [C(\phi) : C(\phi)^\circ] = 2$ if $\phi$ is
relevant for $\GL_{n/2}(\mathbb H)$. 

Thus we checked that for every $\phi \in \Phi (\GL_n (\R)) ,\; \Irr (\mc S_\phi)$ 
parametrizes the equivalence classes of inner forms $G$ of $\GL_n (\R)$ for which $\phi$ is
relevant. To conclude, we apply the LLC for $G$.
\end{proof}

\section{Inner forms of $\SL_n (F)$}
\label{sec:SLn}

As in the previous section, let $D$ be a division algebra over dimension $d^2$ over 
its centre $F$, with reduced norm Nrd$: D \to F$. We write
\[
\GL_m (D)^\sharp := \{ g \in \GL_m (D) : \mr{Nrd}(g) = 1 \} 
\]
Notice that it equals the derived group of $\GL_m (D)$. It is an inner form of
$\SL_{md}(F)$, and every inner form of $\SL_n (F)$ is isomorphic to such a group.
We use the same equivalence relation and parametrization for inner forms of 
$\SL_n (F)$ as for $\GL_n (F)$, as described by \eqref{eq:Kott} and \eqref{eq:5.7}.

As Langlands dual group of $G^\sharp = \GL_m (D)^\sharp$ we take
\[
{}^L G^\sharp = \check G^\sharp = \PGL_n (\C) . 
\]
In particular every Langlands parameter for $G = \GL_m (D)$ gives rise to one for
$G^\sharp$. In line with \cite[\S 10]{Bor}, the L-packets for $G^\sharp$ are 
derived from those for $G$ in the following way. It is known \cite{Wei1} that 
every $\phi^\sharp \in \Phi (G^\sharp)$ lifts to a $\phi \in \Phi (G)$. The L-packet
$\Pi_\phi (G)$ from \eqref{eq:5.2} consists of a single $G$-representation, which
we will denote by the same symbol. Its
restriction to $G^\sharp$ depends only on $\phi^\sharp$, because a different lift
$\phi'$ of $\phi^\sharp$ would produce $\Pi_{\phi'}(G)$ which only differs from
$\Pi_\phi (G)$ by a character of the form 
\[
g \mapsto |\mr{Nrd}(g) |_F^z \quad \text{with } z \in \C . 
\]
We call the restriction of $\Pi_\phi (G)$ to $G^\sharp \; \pi_\phi (G)^\sharp$. 
In general it is reducible, and with it one associates the L-packet
\[
\Pi_{\phi^\sharp} (G^\sharp) := \{ \pi^\sharp \in \Irr (G^\sharp) : \pi^\sharp
\text{ is a constituent of } \pi_\phi (G)^\sharp \} .
\]
The goal of this section is an analogue of Theorem \ref{thm:5.2}.
First we note that every irreducible $G^\sharp$-representation (say $\pi$) is a 
member of an L-packet $\Pi_{\phi^\sharp}(G^\sharp)$, because it appears in a 
$G$-representation (for example in $\mathrm{Ind}_{G^\sharp}^G \pi$). 
Second, by \cite[Lemma 12.1]{HiSa}
two L-packets $\Pi_{\phi_1^\sharp}$ and $\Pi_{\phi_2^\sharp}$ are either disjoint
or equal, and the latter happens if and only if $\phi_1^\sharp$ and $\phi_2^\sharp$
are $\PGL_n (\C)$-conjugate. Thus the main problem is the parametrization of the 
L-packets. Such a parametrization of $\Pi_{\phi^\sharp} (G^\sharp)$ was given
in \cite{HiSa} in terms of S-groups, at least when $F$ has characteristic zero and 
$\phi$ is generic. After recalling this method, we will generalize it. Put
\[
X^G(\Pi_\phi (G)) = \{ \gamma \in \Irr (G / G^\sharp) : \Pi_\phi (G) \otimes \gamma
\cong \Pi_\phi (G) \} .
\]
Notice that every element of $X^G (\Pi_\phi (G))$ is a character, which by Schur's lemma
is trivial on $Z(G)$. Since $G / G^\sharp Z(G)$ is an abelian group and all its elements
have order dividing $n$, the same goes for $X^G (\Pi_\phi (G))$. Moreover 
$X^G (\Pi_\phi (G))$ is finite, as we will see in \eqref{eq:6.1}. On general grounds 
\cite[Lemma 2.4]{HiSa} there exists a 2-cocycle $\kappa_{\phi^\sharp}$ such that 
\begin{equation}\label{eq:6.25}
\C [X^G (\Pi_\phi (G)),\kappa_{\phi^\sharp}] \cong \mathrm{End}_{G^\sharp}(\Pi_\phi (G)) .
\end{equation}
By \cite[Corollary 2.10]{HiSa} the decomposition of $\pi_\phi (G)^\sharp$ as
a representation of $G^\sharp \times X^G (\Pi_\phi (G))$ is
\begin{equation}\label{eq:6.14}
\pi_\phi (G)^\sharp \cong \bigoplus_{\rho \in \Irr (\C [X^G(\Pi_\phi (G)),\kappa_{\phi^\sharp}])}
\Hom_{\C [X^G(\Pi_\phi (G)),\kappa_{\phi^\sharp}]} (\rho,\pi_\phi (G)^\sharp) \otimes \rho .
\end{equation}
The isotropy group of $\phi$ in $C(\phi^\sharp)$ is
\[
C(\phi) = Z(\SL_n (\C)) C(\phi)^\circ = Z(\SL_n (\C)) C(\phi^\sharp)^\circ .
\]
We also note that
\begin{equation}\label{eq:6.2}
\begin{aligned}
& C(\phi^\sharp) / C(\phi) \cong \mc S_{\phi^\sharp} / \mc Z_{\phi^\sharp} , \quad
\text{where} \\
& \mc Z_{\phi^\sharp} = Z(\SL_n (\C)) C(\phi^\sharp)^\circ / C(\phi^\sharp)^\circ \cong
Z(\SL_n (\C)) / Z(\SL_n (\C)) \cap C(\phi^\sharp)^\circ .
\end{aligned}
\end{equation}
Assume for the moment that $D \not\cong \mathbb H$, so Nrd$: D \to F$ is surjective by
\cite[Proposition X.2.6]{Wei2}. Let 
$\hat \gamma : \mathbf W_F \to \C^\times \cong Z(\GL_n (\C))$ correspond to 
$\gamma \in \Irr (F^\times) \cong \Irr (G / G^\sharp)$
via local class field theory. By the LLC for $G ,\; \phi$ is 
$\GL_n (\C)$-conjugate to $\phi \hat \gamma$ for all $\gamma \in X^G (\Pi_\phi (G))$.
As $(\phi \hat \gamma )^\sharp = \phi^\sharp ,\; \phi$ and $\phi \hat \gamma$ are in fact
conjugate by an element of $C(\phi^\sharp) \subset \SL_n (\C)$. 
This gives an isomorphism 
\begin{equation}\label{eq:6.1}
C(\phi^\sharp) / C(\phi) \cong X^G (\Pi_\phi (G)) , 
\end{equation}
showing in particular that the left hand side is abelian. 
Since $C(\phi^\sharp) / C(\phi)$
is the component group of the centralizer of the subset im$(\phi^\sharp)$ of the algebraic
group $\PGL_n (\C)$, the groups in \eqref{eq:6.1} are finite.
Thus we obtain a central extension of finite groups
\begin{equation}\label{eq:6.4}
1 \to \mc Z_{\phi^\sharp} \to \mc S_{\phi^\sharp} \to X^G (\Pi_\phi (G)) \to 1 .
\end{equation}
The algebra \eqref{eq:6.25} can be described with the idempotent
\[
e_{\chi_G} := |\mc Z_{\phi^\sharp}|^{-1} \sum\nolimits_{z \in \mc Z_{\phi^\sharp}}
\chi_G (z^{-1}) z \; \in \; \C [\mc Z_{\phi^\sharp}] .
\]

\begin{thm}\label{thm:6.3}
Let $G = \GL_m (D)$ with $D \not\cong \mathbb H$.
There exists an isomorphism 
\[
\C [X^G (\Pi_\phi (G)),\kappa_{\phi^\sharp}] = \C [\mc S_{\phi^\sharp} / \mc Z_{\phi^\sharp},
\kappa_{\phi^\sharp}] \cong e_{\chi_G} \C[ \mc S_{\phi^\sharp}]
\]
such that for any $s \in \mc S_{\phi^\sharp}$ the subspaces $\C s \mc Z_{\phi^\sharp}$
on both sides correspond. Moreover any two such isomorphisms differ only by
a character of $\mc S_{\phi^\sharp} / \mc Z_{\phi^\sharp}$.
\end{thm}
\begin{proof}
\emph{(of the case $\char(F) = 0$.)}\\
First we suppose that char$(F) = 0$ and that the representation $\Pi_\phi (G)$
is tempered. In the archimedean case the cocycle $\kappa_{\phi^\sharp}$ is trivial by 
\cite[Lemma 3.1 and page 69]{HiSa}. In the non-archimedean case the theorem is a 
reformulation of \cite[Lemma 12.5]{HiSa}. We remark that this is a deep result, 
its proof makes use of endoscopic transfer and global arguments. 

Consider a possibly unbounded Langlands parameter $\phi^\sharp \in \Phi (G^\sharp)$, 
with a lift $\phi \in \Phi (G)$. Let $Y$ be a connected set of unramified twists 
$\phi_\chi$ of $\phi$, such that $C(\phi_\chi) = C(\phi)$ and 
$C(\phi_\chi^\sharp) = C(\phi^\sharp)$ for all $\phi_\chi \in Y$. It is easily seen that 
we can always arrange that $Y$ contains bounded Langlands parameters. The reason is that 
for any element (here the image of a Frobenius element of $\mathbf{W}_F$ under $\phi$)
 of a torus in a complex reductive group, there is an element 
of the maximal compact subtorus which has the same centralizer.

The construction of the intertwining operators 
\begin{equation}
I_\gamma \in \Hom_G (\Pi_\phi (G),\Pi_\phi (G) \otimes \gamma) \text{ for }
\gamma \in X^G (\Pi_\phi (G)) 
\end{equation}
is similar to that for R-groups. It determines the cocycle by
\[
I_\gamma I_{\gamma'} = \kappa_{\phi^\sharp}(\gamma,\gamma') I_{\gamma \gamma'} .
\]
The $I_\gamma$ can be
chosen independently of $\chi \in X_{nr}(M)$, so the $\kappa_{\phi_\chi^\sharp}$ do
not depend on $\chi$. For $\phi_\chi^\sharp$ tempered we already have the required
algebra isomorphisms, and now they extend by constancy to all $\phi_\chi^\sharp \in Y$.
This concludes the proof in the case char$(F) = 0$.
\end{proof}

The proof of the case char$(F) > 0$ requires more techniques, we complete it in
Section \ref{sec:CL}. 

\smallskip

For a character $\chi$ of $\mc Z_{\phi^\sharp}$ or of $Z(\SL_n (\C))$ we write
\begin{equation}\label{eq:6.5}
\Irr (\mc S_{\phi^\sharp},\chi) := \Irr \big( e_\chi \C[ \mc S_{\phi^\sharp}] \big) =
\{ (\pi,V) \in \Irr (\mc S_{\phi^\sharp}) : 
\mc Z_{\phi^\sharp}  \text{ acts on } V \text{ as } \chi \} .
\end{equation}
We will use this with the characters $\chi_G = \chi_{G^\sharp}$ from Lemma \ref{lem:5.1}. 

We still assume that $D \not\cong \mathbb H$. As shown in \cite[Corollary 2.10]{HiSa}, 
the isomorphism \eqref{eq:6.25} and Theorem \ref{thm:6.3} imply that 
\begin{equation}\label{eq:6.7}
\pi (\phi^\sharp,\rho) := \mathrm{Hom}_{\mc S_{\phi^\sharp}}(\rho,\Pi_\phi (G))
\end{equation}
defines an irreducible $G^\sharp$-representation for every $\rho \in 
\Irr (\mc S_{\phi^\sharp},\chi_{G^\sharp})$. 
In general $\pi (\phi^\sharp,\rho)$ is
not canonical, it depends on the choice of an algebra isomorphism as in Theorem 
\ref{thm:6.3}. Hence the map $\rho \mapsto \pi (\phi^\sharp,\rho)$ is canonical up to
an action of 
\[
\Irr \big( \mc S_{\phi^\sharp} / \mc Z_{\phi^\sharp} \big) \cong 
\Irr \big( X^G (\Pi_\phi (G)) \big)
\]
on $\Irr \big( e_{\chi_G} \C[ \mc S_{\phi^\sharp}] \big)$. Via \eqref{eq:6.14} and
Theorem \ref{thm:6.3} this corresponds to an action of $\Irr \big( X^G (\Pi_\phi (G)) \big)$ 
on $\Pi_{\phi^\sharp}(G)$, which can be described explicitly.
Since $X^G (\Pi_\phi (G))$ is a subgroup of $\Irr \big( G / G^\sharp Z(G) \big)$,
$\Irr \big( X^G (\Pi_\phi (G)) \big)$ is a quotient of 
$G / G^\sharp Z(G)$, say $G / H$ for some $H \supset G^\sharp Z(G)$. 
This means that every $c \in \Irr \big( X^G (\Pi_\phi (G)) \big)$ determines a coset 
$g_c H$ in $G$. Now the formula
\begin{equation}\label{eq:6.3}
c \cdot \pi = g_c \cdot \pi \text{, where } (g_c \cdot \pi)(g) = \pi (g_c^{-1} g g_c) 
\end{equation}
defines the action of $\Irr \big( X^G (\Pi_\phi (G)) \big)$ on 
$\Pi_{\phi^\sharp}(G^\sharp)$. In other words, the representation $\pi (\phi^\sharp, \rho)
\in \Pi_{\phi^\sharp}(G^\sharp)$ is canonical up to the action of $G$ on 
$G^\sharp$-representations.

For $D = \mathbb H$ some modifications must be made. In that case $G = G^\sharp Z(G)$, so 
$\mathrm{Res}_{G^\sharp}^{G}$ preserves irreducibility of representations and 
$X^G (\Pi_\phi (G)) = 1$. Moreover $G / G^\sharp \cong \R^\times_{>0} \not\cong \R^\times$,
which causes \eqref{eq:6.1} and \eqref{eq:6.4} to be invalid for $D = \mathbb H$.
However, \eqref{eq:6.5} still makes sense, so we define
\begin{equation}\label{eq:6.6}
\pi (\phi^\sharp,\rho) := \Pi_\phi (\GL_m (\mathbb H)) \quad \text{for all} \quad
\rho \in \Irr \big( \mc S_{\phi^\sharp},\chi_{\mathbb H^\times} \big) .
\end{equation}

As mentioned before, Hiraga and Saito \cite{HiSa} have established the local Langlands
correspondence for irreducible generic representations of inner forms of $\SL_n (F)$,
where $F$ is a local field of characteristic zero. We will generalize this on the one 
hand to local fields $F$ of arbitrary characteristic and on the other 
hand to all irreducible admissible representations. We will do so for all inner forms
of $\SL_n (F)$ simultaneously, to obtain an analogue of Theorem \ref{thm:5.2}.

Like for $\GL_n (F)$ we define
\[
\Phi^e (\text{inn } \SL_n (F)) = \big\{ (\phi^\sharp ,\rho) : \phi^\sharp \in
\Phi (\SL_n (F)), \rho \in \Irr (\mc S_{\phi^\sharp}) \big\} .
\]
Notice that the restriction of $\rho$ to 
$\mc Z_{\phi^\sharp} \cong Z(\SL_n (\C)) / Z(\SL_n (\C)) \cap C(\phi^\sharp)^\circ$
determines an inner form $G_\rho$ of $\GL_n (F)$ (up to isomorphism) via \eqref{eq:Kott} 
and Lemma \ref{lem:5.1}. Its derived group $G_\rho^\sharp$ is the inner form of 
$\SL_n (F)$ associated to $\rho$. 

We note that the actions of $\PGL_n (\C)$ on the various $\Phi^e (G^\sharp)$ combine to
an action on $\Phi^e (\text{inn } \SL_n (F))$. With the collection of equivalence classes 
$\Phi^e (\text{inn } \SL_n (F))$ we 
can formulate the local Langlands correspondence for all such inner forms simultaneously.

First we consider the non-archimedean case.
As for $\GL_n (F)$, we must fix one group in every equivalence class of inner forms.
We choose the groups $\GL_m ( [L/F,\chi,\varpi_F] )^\sharp$ with $[L/F,\chi,\varpi_F]$ as 
in \eqref{eq:5.6}, and call these the standard inner forms of $\SL_n (F)$.

\begin{thm}\label{thm:6.5}
Let $F$ be a non-archimedean local field. There exists a bijection
\[
\begin{array}{lll}
\Phi^e (\mr{inn } \; \SL_n (F)) & \to & \big\{ (G^\sharp,\pi) : G^\sharp \text{ standard
inner form of } \SL_n (F), \pi \in \Irr (G^\sharp) \big\} \\
(\phi^\sharp,\rho) & \mapsto & \big( G_\rho^\sharp,\pi (\phi^\sharp,\rho) \big) 
\end{array}
\]
with the following properties:
\enuma{
\item Suppose that $\rho$ sends $\exp(2 \pi i/n) \in Z(\SL_n (\C))$ to a primitive 
$d$-th root of unity $z$. Then $G_\rho^\sharp = \GL_m ( [L/F,\chi,\varpi_F] )^\sharp$, where 
$md = n$ and $\chi : \mr{Gal}(L/F) \to \C^\times$ sends the Frobenius automorphism to $z$.
\item Suppose that $\phi^\sharp$ is relevant for $G^\sharp$ and lifts to $\phi \in \Phi (G)$.
Then the restriction of $\Pi_\phi (G)$ to $G^\sharp$ is $\bigoplus_{\rho \in 
\Irr (\mc S_{\phi^\sharp},\chi_{G^\sharp})} \pi (\phi^\sharp,\rho) \otimes \rho$.
\item $\pi (\phi^\sharp,\rho)$ is essentially square-integrable if and only if 
$\phi^\sharp (\mathbf W_F \times \SL_2 (\C))$ is not contained in any proper parabolic
subgroup of $\PGL_n (\C))$.
\item $\pi (\phi^\sharp,\rho)$ is tempered if and only if $\phi^\sharp$ is bounded.
}
\end{thm}
\begin{proof}
Let $\phi^\sharp \in \Phi (\SL_n (F))$ and lift it to $\phi \in \Phi (\GL_n (F))$. Then
$C(\phi^\sharp)^\circ = C(\phi)^\circ$ and $\mc Z_{\phi^\sharp} = \mc Z_\phi$, so by 
Lemma \ref{lem:5.1} the set of standard
inner forms of $\SL_n (F)$ for which $\phi^\sharp$ is relevant is in natural bijection with
\[
\Irr (\mc Z_{\phi^\sharp}) = 
\Irr \big( Z(\SL_n (\C)) / Z(\SL_n (\C)) \cap C(\phi^\sharp)^\circ \big) .
\]
Hence the collection of $(\phi^\sharp,\rho) \in \Phi^e (\text{inn } \SL_n (F))$ with 
$\phi^\sharp$ fixed is
\begin{equation}\label{eq:6.8}
\big\{ (\phi^\sharp,\rho) : \rho \in \Irr (\mc S_{\phi^\sharp},\chi_{G^\sharp}) 
\text{ with } \phi^\sharp \text{ relevant for } G^\sharp \big\} .
\end{equation}
Thus (a) automatically holds. Part (b) is a consequence of \eqref{eq:6.25}
and Theorem \ref{thm:6.3}, see \cite[Corollary 2.10]{HiSa}. Together with the remarks
at the beginning of the section this shows that the map from the theorem is bijective.

For part (d), it is clear that the restriction of a tempered $G$-representation to
$G^\sharp$ is still tempered. Hence $\pi (\phi^\sharp,\rho)$ is tempered if $\phi^\sharp$
has a bounded lift $\phi$, that is, if $\phi^\sharp$ is itself bounded. Conversely,
if $\pi (\phi^\sharp,\rho)$ is tempered, then all its matrix coefficients are tempered
on $G^\sharp$. Lift $\phi^\sharp$ to $\phi \in \Phi (G)$ such that the central
character of $\Pi_\phi (G)$ is unitary. Since $\pi (\phi^\sharp,\rho)$ generates 
$\Pi_\phi (G)$ as a $G$-representation, all matrix coefficients of $\pi_\phi (G)$
are tempered on $G^\sharp Z(G)$. As $G^\sharp Z(G)$ is of finite index in $G$, this
implies that $\Pi_\phi (G)$ is tempered. One of the properties of the LLC for $G$
says that temperedness of $\Pi_\phi (G)$ is equivalent to boundedness of $\phi$.
Therefore $\phi^\sharp$ is also bounded.

An analogous argument applies to (essentially) square-integrable representations.
These arguments prove parts (c) and (d). 
\end{proof}

Let us discuss an archimedean analogue of Theorem \ref{thm:6.5}, that is,
for the groups $\SL_n (\C), \SL_n (\R)$ and $\SL_m (\mathbb H)$. In view \eqref{eq:6.6}
we cannot expect a bijection, and part (b) has to be adjusted.

\begin{thm}\label{thm:6.6}
Let $F$ be $\R$ or $\C$. There exists a canonical surjection
\[
\begin{array}{lll}
\Phi^e (\mr{inn } \; \SL_n (F)) & \to & \big\{ (G^\sharp,\pi) : G^\sharp \text{ standard
inner form of } \SL_n (F), \pi \in \Irr (G^\sharp) \big\} \\
(\phi^\sharp,\rho) & \mapsto & \big( G_\rho^\sharp,\pi (\phi^\sharp,\rho) \big) 
\end{array}
\]
with the following properties:
\enuma{
\item The preimage of $\Irr (\SL_n (F))$ consists of the $(\phi^\sharp,\rho)$ with
$\mc Z_{\phi^\sharp} \subset \ker \rho$, and the map is injective on this domain.
The preimage of $\Irr (\SL_{n/2}(\mathbb H))$ consists of the $(\phi^\sharp,\rho)$
such that $\rho$ is not trivial on $\mc Z_{\phi^\sharp}$, and the map is two-to-one
on this domain.
\item Suppose that $\phi^\sharp$ is relevant for $G^\sharp = \SL_m (D)$ and lifts to 
$\phi \in \Phi (G)$. Then the restriction of $\Pi_\phi (G)$ to $G^\sharp$ is irreducible
if $D = \C$ or $D = \mathbb H$, and is isomorphic to $\bigoplus_{\rho \in 
\Irr (\mc S_{\phi^\sharp} / \mc Z_{\phi^\sharp})} \pi (\phi^\sharp,\rho) \otimes \rho$
in case $D = \R$.
\item $\pi (\phi^\sharp,\rho)$ is essentially square-integrable if and only if 
$\phi^\sharp (\mathbf W_F)$ is not contained in any proper parabolic
subgroup of $\PGL_n (\C))$.
\item $\pi (\phi^\sharp,\rho)$ is tempered if and only if $\phi^\sharp$ is bounded.
}
\end{thm}
\begin{proof}
Theorem \ref{thm:5.2} and the start of the proof of Theorem \ref{thm:6.5} show that 
\eqref{eq:6.8} is also valid in the archimedean case. To see that the map thus obtained
is canonical, we will of course use that the LLC for $\GL_m (D)$ is so. For $\SL_n (F)$
the intertwining operators admit a canonical normalization in terms of Whittaker functionals
\cite[pages 17 and 69]{HiSa}, so the definition \eqref{eq:6.7} of $\pi (\phi^\sharp,\rho)$
can be made canonical. For $\SL_m (\mathbb H)$ the definition \eqref{eq:6.6} clearly leaves
no room for arbitrary choices.

Part (a) and part (b) for $D = \R$ follow as in the non-archimedean case, except that for 
$D = \mathbb H$ the preimage of $\pi (\phi^\sharp,\rho)$ is in bijection with $\Irr \big( 
\mc S_\phi, e_{\mathbb H^\times} \big)$. To prove part (b) for $D = \C$ and $D = \mathbb H$,
it suffices to remark that $\mathrm{Res}_{G^\sharp}^G$ preserves irreducibility, as
$G = G^\sharp Z(G)$. The proof of part (c) and (d) carries over from Theorem \ref{thm:6.5}.

It remains to check that the map is two-to-one on $\Irr (\SL_m (\mathbb H))$. For this we have
to compute
\begin{equation}\label{eq:6.9}
\mc S_{\phi^\sharp} / \mc Z_{\phi^\sharp} = C(\phi^\sharp) / C(\phi^\sharp)^\circ = 
C(\phi^\sharp) / C(\phi) .
\end{equation}
Consider $\phi^\sharp \in \Phi (\SL_m (\mathbb H))$ with two lifts $\phi,\phi' \in \Phi 
(\GL_m (\mathbb H))$ that are conjugate under $\GL_{2m}(\C)$. The restriction of 
$\phi^{-1} \phi'$ to $\C^\times \subset \mathbf{W}_\R$ is a group homomorphism 
$c : \C^\times \to Z(\GL_{2m}(\C))$. Clearly $\phi$ and $\phi'$ can only be conjugate if $c=1$,
so $\phi'$ can only differ from $\phi$ on $\tau \in \mathbf{W}_\R$. Since
\[
\phi' (\tau)^2 = \phi' (-1) = \phi (-1) = \phi (\tau)^2 ,
\]
either $\phi' (\tau) = - \phi (\tau)$ or $\phi' = \phi$. Recall the standard form of $\phi$
exhibited in the proof of Theorem \ref{thm:5.2}, with image in the Levi subgroup
$\GL_2 (\C)^m$ of $\GL_{2m}(\C)$. It shows that the Langlands parameter $\phi'$ determined by
$\phi' (\tau) = - \phi (\tau)$ is always conjugate to $\phi$, for example by the element
diag$(1,-1,1,\ldots,-1) \in \GL_{2m}(\C)$. Therefore \eqref{eq:6.9} has precisely two elements.
Now $e_{\mathbb H^\times} \C [\mc S_{\phi^\sharp}]$ is a two-dimensional semisimple $\C$-algebra,
so it is isomorphic to $\C \oplus \C$. We conclude that $\Irr ( \mc S_{\phi^\sharp},
e_{\mathbb H^\times})$ has two elements, for every $\phi^\sharp \in \Phi (\SL_m (\mathbb H))$.
\end{proof}

\section{Characterization of the LLC for some representations of $\GL_n (F)$}

It is known from \cite{Hen1} that generic representations of $\GL_n (F)$ can
be characterized in terms of $\gamma$-factors of pairs, where other part of the pair 
is a representation of a smaller general linear group. We will establish a more precise
version for irreducible representations that have nonzero vectors fixed under a 
specific compact open subgroup.

Let $F_s$ be a separable closure of $F$ and let Gal$(F_s / F)^l$ be the 
$l$-th ramification group of Gal$(F_s / F)$, with respect to the upper numbering. 
We define
\[ 
\Phi_l (G) := \{ \phi \in \Phi (G) :  \text{Gal}(F_s / F)^l \subset \ker (\phi) \} .
\]
Notice that 
\[
\Phi_{l'}(G)\subset\Phi_l(G),\qquad\text{if $l'\le l$.}
\]
We will say that $\phi\in\Phi(\GL_n(F))$ is \emph{elliptic} if
its image is not contained in any proper Levi subgroup of $\GL_n (\C)$.

\begin{lem} \label{lem:phi}
Let $\phi\in\Phi(\GL_n(F))$ such that $\phi$ is elliptic and $\SL_2(\C)\subset \ker
(\phi)$. Then we have
\[\text{$\phi\notin\Phi_{d(\phi)}(\GL_n(F))$ and
$\phi\in\Phi_{l}(\GL_n(F))$ for any $l>d(\phi)$,}\]
where  \begin{equation} \label{eqn:depth}
d(\phi):=\begin{cases}
0&\text{if $\bI_F\subset\ker(\phi)$,}\cr
\frac{\displaystyle \swan(\phi)}{\displaystyle n}&\text{otherwise,}
\end{cases}\end{equation}
and $\swan(\phi)$ denotes the Swan conductor of $\phi$.
\end{lem}
\begin{proof}
Let $c(\phi)$ denote the greatest integer such that
\[\Gal(F_s/F)_{c(\phi)+1}\subset\ker(\phi),\]
if $\bI_F=\Gal(F_s/F)_0$ is not contained in $\ker(\phi)$,
and $-1$ otherwise. Recall the Herbrand function $\varphi_{F_s/F}$
\cite[Chap.~IV, \S~3]{Ser} that allows us to pass from the lower number to
the upper ones:
\[
\Gal(F_s/F)_l=\Gal(F_s/F)^{\varphi_{F_s/F}(l)}.
\]
Let $a(\phi)$ denote the Artin conductor of $\phi$. Because $\phi$ is assumed 
to be elliptic, the restriction of $\phi$ to $\bW_F$ is irreducible. The equality
\[
a(\phi)=n\left(\varphi_{F_s/F}(c(\phi))+1\right)
\]
was shown for $n=1$ in \cite[Chap.~VI, \S~2, Proposition~5]{Ser}. The proof 
for arbitrary $n$ is similar, see \cite[\S~2]{GrRe}. By the very definition
of the Swan conductor
\[
\varphi_{F_s/F} (c(\phi))=\frac{a(\phi)}{n}-1=\frac{\swan(\phi)}{n}.
\]
Then it follows from the definition of $c(\phi)$ that
$d(\phi)$ is the greatest number such that
\[\Gal(F_s/F)^{d(\phi)}\not\subset\ker(\phi).\] Hence we have 
\[\phi\notin\Phi_{d(\phi)}(\GL_n(F))\quad\text{and}\quad
\phi\in\Phi_{d(\phi)+1}(\GL_n(F)). \qedhere
\]
\end{proof} 

Let $\mf A$ be a hereditary $\mf o_F$-order $\mf A$ in $\Mat_n(F)$.
Let $\mf P$ denote the Jacobson
radical of $\mf A$, and let $e(\mf A)$ denote the $\mf o_F$-period of
$\mf A$, that is, the integer $e$ defined by
$\mf p_F\mf A=\mf P^e$.
Define a sequence of compact open subgroups of
$\GL_n(F)$ by
\[U^0(\mf A)=\mf A^\times,\quad\text{and}\quad
U^m(\mf A)=1+\mf P^m,\;\;m\ge 1.\] 
Let $m$, $m'$ be integers satisfying $m>m'\ge\lfloor m/2\rfloor$. There is
a canonical isomorphism 
\[U^{m'+1}(\mf A)/U^{m+1}(\mf A)\to\mf P^{m'+1}/\mf P^{m+1},\]
given by $x\mapsto x-1$. This leads to an isomorphism from 
$\mf p^{-1}\mf P^{-m}/\mf p^{-1}\mf P^{-m'}$ to the Pontrjagin dual of 
$U^{m'+1}(\mf A)/U^{m+1}(\mf A)$,  explicitly given by
\[\beta+\mf p^{-1}\mf P^{-m'}\mapsto\psi_\beta \qquad \beta\in\mf p^{-1}
\mf P^{-m},\]
with $\psi_\beta(1+x)=(\psi\circ\tr_{\Mat_n(F)})(\beta x)$, for $x\in \mf P^{-m'}$.
 
We recall from \cite[(1.5)]{BKlivre} that a \emph{stratum} is a quadruple
$[\mf A,m,m',\beta]$ consisting of a hereditary $\mf o_F$-order $\mf A$ in
$\Mat_n(F)$, integers $m>m'\ge 0$, and an element $\beta\in\Mat_n(F)$ with
$\mf A$-valuation $\nu_{\mf A}(\beta)\ge -m$.
A stratum of the form $[\mf A,m,m-1,\beta]$ is called \emph{fundamental}
\cite[(2.3)]{BKlivre} if the coset $\beta+\mf p^{-1}\mf P^{1-m}$ does not 
contain a nilpotent element of $\Mat_n(F)$. 
We remark that the formulation in \cite{Bus} is slightly different because the
notion of a fundamental stratum there allows $m$ to be $0$. 

Fix an irreducible supercuspidal representation $\pi$ of $\GL_n(F)$. 
According to \cite[Theorem~2]{Bus} there exists a hereditary order $\mf A$ 
in $\Mat_n(F)$ such that either
\enuma{
\item $\pi$ contains the trivial character of $U^1(\mf A)$, or
\item there is a fundamental stratum $[\mf A,m,m-1,\beta]$ in $\Mat_n(F)$ such
that $\pi$ contains the character $\psi_\beta$ of $U^m(\mf A)$.
}

Moreover, in case (b), if a stratum $[\mf A_1,m_1,m_1-1,\beta_1]$
is such that $\beta_1$ occurs in the restriction of $\pi$ to
$U^{m_1}(\mf A_1)$, then $m_1/e(\mf A_1)\ge m/e(\mf A)$, and we have equality
here if and only $[\mf A_1,m_1,m_1-1,\beta_1]$ is fundamental \cite[Theorem~$2'$]{Bus}. 

The above provides a useful invariant of the representation, called the \emph{depth} 
(or \emph{normalized level}) of $\pi$. It is defined as
\begin{equation} \label{eqn:def_depth}
d(\pi):=\min\left\{m/e(\mf A)\right\},
\end{equation}
where $(m,\mf A)$ ranges over all pairs consisting of an integer $m\ge 0$
and a hereditary $\mf o_F$-order in $\Mat_n(F)$ such that $\pi$ contains 
the trivial character of $U^{m+1}(\mf A)$. 

The following result was claimed in \cite[Theorem 2.3.6.4]{Yu}. Although Yu did not
provide a proof, he indicated that an argument along similar lines as ours is possible.

\begin{prop} \label{lem:pi}
Let $\pi \in \Irr (\GL_n (F)$ be supercuspidal and $\phi:=\rec_{F,n}(\pi)$.  Then 
\[d(\phi)=d(\pi).\]  
\end{prop}
\begin{proof} 
We have
\begin{equation} \label{eqn:epsA}
\vareps(s,\phi,\psi) =
\vareps(0,\phi,\psi) \, q^{-a(\phi)s} \quad\text{with } \vareps(0,\phi,\psi)
\in \C^\times .
\end{equation}
It is known that the LLC for $\GL_n (F)$ preserves the $\vareps$-factors:
\[
\vareps(s,\phi,\psi)=\vareps(s,\pi,\psi),
\]
where $\vareps(s,\pi,\psi)$ is the Godement-Jacquet local constant \cite{GoJa}. 
It takes the form
\begin{equation} \label{eqn:epsGJ}
\vareps(s,\pi,\psi) =
\vareps(0,\pi,\psi) \, q^{-f(\pi)s},\quad\text{where $\vareps(0,\pi,\psi)
\in \C^\times$.}
\end{equation} 
Recall that $f(\pi)$ is an integer, called the conductor of $\pi$.
It follows from \eqref{eqn:epsA} and \eqref{eqn:epsGJ} that
\begin{equation} \label{eqn:a=f}
a(\phi)=f(\pi).
\end{equation}
In the case when $\pi$ is an unramified representation of $F^\times$, 
the inertia subgroup $\bI_F$ is contained in $\ker \phi$,
with $\phi=\rec_{F,1}(\pi)$. Hence (\ref{eqn:depth}) implies that $a(\phi)=0$.
On the other hand, $\pi$ is trivial on $\mf o_F^\times$, and a fortiori
trivial on $1+\mf p_F=U^1(\mf A)$, with $\mf A=\mf o_F$. Then
(\ref{eqn:def_depth}) implies that $d(\pi)=0=d(\phi)$.

From now on we will assume that we are not in the above special case, that is,
we assume that $n\ne 1$ or that $\pi$ is ramified.
Let $\mf A$ be a principal $\mf o_F$-order in
$\Mat_n(F)$ such that $e(\mf A)=n/\gcd(n,f(\pi))$, and
let $\mf K(\mf A)$ denote the normalizer in $\GL_n(F)$ of $\mf A$.
By \cite[Theorem~3]{Bus} the restriction of $\pi$ to $\mf K(\mf A)$ 
contains a \emph{nondegenerate} (in the sense of \cite[(1.21)]{Bus}) 
representation $\varrho$ of $\mf K(\mf A)$, and we have \cite[(3.7)]{Bus}
\begin{equation} \label{eqn:Bus}
d(\varrho)=e(\mf A)\left(\frac{f(\pi)}{n}-1\right),
\end{equation} 
where $d(\varrho)\ge 0$ is the least integer such that 
\[U^{d(\varrho)+1}(\mf A)\subset\ker(\varrho).\]
Moreover, if the irreducible representation $\varrho'$ of $\mf K(\mf A)$
occurs in the restriction of $\pi$ to $\mf K(\mf A)$, then $d(\varrho')=d(\varrho)$ 
if and only if $\varrho'$ is nondegenerate \cite[(5.1)~(iii)]{Bus}.
Hence we obtain from \eqref{eqn:a=f} and \eqref{eqn:Bus} that
\begin{equation} \label{eqn:drho}
\frac{d(\varrho')}{e(\mf A)}=\frac{f(\pi)}{n}-1=\frac{a(\phi)}{n}-1= d(\phi) 
\end{equation}
for every nondegenerate irreducible representation $\rho'$ of $\mf K(\mf A)$ which
occurs in the restriction of $\pi$ to $\mf K(\mf A)$.

It follows from the definition \eqref{eqn:def_depth} of $d(\pi)$, that
\begin{equation} \label{eqn:ineq}
d(\pi)\le \frac{d(\varrho')}{e(\mf A)},
\end{equation}
for every nondegenerate irreducible representation $\rho'$ of $\mf K(\mf A)$ which
occurs in the restriction of $\pi$ to $\mf K(\mf A)$.

We will check that \eqref{eqn:ineq} is actually an equality.
The case where $d(\pi)=0$ is easy, so we only consider $d(\pi)>0$.

Let $\mf A'$ be any hereditary $\mf o_F$-order $\mf A'$ in $\Mat_n(F)$,
and define $m_{\mf A'}(\pi)$ to be the least non-negative integer $m$ such
that the restriction of $\pi$ to $U^{m+1}(\mf A')$ contains the trivial
character. Then choose $\mf A'$ so that $m_{\mf A'}(\pi)/e(\mf A')$ is
minimal, and let $[\mf A', m_{\mf A'}(\pi),m_{\mf A'}(\pi)-1,\beta]$ be a
stratum occuring in $\pi$. By \cite[Theorem~$2'$]{Bus} this is a fundamental
stratum. By \cite[(3.4)]{Bus} we may assume that the integers $e(\mf A')$
and $m_{\mf A'}(\pi)$ are relatively prime. Hence we may 
apply \cite[(3.13)]{Bus}. We find that $\mf A'$ is principal and that
every irreducible representation $\varrho$ of $\mf K(\mf A')$ which occurs
in the restriction of $\pi$ to $\mf K(\mf A')$, and such that the
restriction of $\varrho$ to $U^{m_{\mf A'}(\pi)}(\mf A')$ contains $\psi_\beta$,
is nondegenerate. In particular we have $d(\varrho')=m_{\mf A'}(\pi)$.

It remains to check that the principal order $\mf A'$ satisfies 
\begin{equation} \label{eqn:check}
e(\mf A')=n/\gcd(n,f(\pi)).
\end{equation}
Let $b=\gcd(n,f(\pi))$. Set 
$n=n'b$ and $f(\pi)=f'(\pi)b$. 
By using \cite[(3.9)]{Bus}, we obtain that $n'$ divides $e(\mf A')$. 
Let $\mf P'$ denote the Jacobson radical of $\mf A'$.
Then \cite[(3.3.8)]{BuFr} and \cite[(3.8)]{Bus} assert that
\[q^{f(\pi)}=[\mf A':\mf p_F(\mf P^{\prime})^{d(\varrho')}]^{1/n}.\]
That is, since $\mf p_F\mf A'=(\mf P')^{e(\mf A')}$,
\[q^{f(\pi)}=[\mf A':(\mf P^{\prime})^{d(\varrho')+e(\mf A')}]^{1/n}
=q^{n(d(\varrho')+e(\mf A'))/e(\mf A')}=
q^{n(1+d(\varrho')/e(\mf A'))}.\]
Hence we get
\[f(\pi)=n(1+d(\varrho')/e(\mf A')),\]
that is,
\[d(\varrho')=\frac{e(\mf A')f(\pi)}{n}-e(\mf A')
=\frac{e(\mf A')f'(\pi)}{n'}-e(\mf A').\]
Hence we have
\[n'd(\varrho')=e(\mf A')f'(\pi)-e(\mf A')n'.\]
Since $e(\mf A')$ and $d(\rho')=m_{\mf A'}(\pi)$ are relatively prime, we
deduce that $e(\mf A')$ divides $n'$. Thus we have $e(\mf A')=n'$, which 
means that \eqref{eqn:check} holds. 

We conclude that \eqref{eqn:ineq} is indeed an equality, which together with
\eqref{eqn:drho} shows that $d(\varrho')=d(\pi)$.
\end{proof} 

As congruence subgroups are the main examples of groups like $U^m (\mf A)$ above,
they have a link with depths. This can be made precise.
Let $K_0 = \GL_n (\mf o_F)$ be the standard maximal compact subgroup of $\GL_n (F)$ 
and define, for $r \in \Z_{>0}$:
\[
K_r = \ker \big( \GL_n (\mf o_F) \to \GL_n (\mf o_F / \mf p_F^r) \big) =
1 + \Mat_n (\mf p_F^r) .
\]
We denote the set of irreducible smooth $\GL_n (F)$-representations that are generated
by their $K_r$-invariant vectors by $\Irr (\GL_n (F),K_r)$. To indicate the ambient
group $\GL_n (F)$ we will sometimes denote $K_r$ by $K_{r,n}$.

\begin{lem}\label{lem:depthKr}
For $\pi \in \Irr (\GL_n (F))$ and $r \in \Z_{> 0}$ the following are equivalent:
\begin{itemize}
\item $\pi \in \Irr (\GL_n (F), K_r)$,
\item $d(\pi) \leq r-1$.
\end{itemize}
\end{lem}
\begin{proof}
For this result it is convenient to use the equivalent definition of depth provided by
Moy and Prasad \cite{MoPr}. In their notation the group $K_r$ is $P_{o,(r-1)+}$,
where $o$ denotes the origin in the standard apartment of the Bruhat--Tits building of 
$\GL_n (F)$. From the definition in \cite[\S 3.4]{MoPr} we read off that any 
$\pi \in \Irr (\GL_n (F), K_r)$ has depth $\leq r-1$.

Conversely, suppose that $d(\pi) \leq r-1$. Then $\pi$ has nonzero vectors fixed by the
group $P_{x,(r-1)+}$, where $x$ is some point of the Bruhat--Tits building. Since we may
move $x$ within its $\GL_n (F)$-orbit and there is only one orbit of vertices, we may assume 
that $x$ lies in the star of $o$. As $r-1 \in Z_{\geq 0}$, there is an inclusion
\[
P_{x,(r-1)+} \supset P_{o,(r-1)+} = K_r ,
\]
so $\pi$ has nonzero vectors fixed by $K_r$.
\end{proof}

Let us recall some basic properties of generic representations, from
\cite[Section 2]{JPS2}. Let $\psi \colon F \to \C^\times$ be a character which is
trivial on $\mf o_F$ but not on $\varpi_F^{-1} \mf o_F$. We note that $\psi$ is
unitary because $F / \mf o_F$ is a union of finite subgroups. Let $U = U_n$ be the standard
unipotent subgroup of $\GL_n (F)$, consisting of upper triangular matrices. We need
a character $\theta$ of $U$ which does not vanish on any of the root subgroups
associated to simple roots. Any choice is equally good, and it is common to take
\[
\theta \big( (u_{i,j})_{i,j=1}^n \big) = 
\psi \Big (\sum\nolimits_{i=1}^{n-1} u_{i,i+1} \Big) . 
\]
Let $(\pi,V) \in \Irr (\GL_n (F))$. One calls $\pi$ generic if there exists a
nonzero linear form $\lambda$ on $V$ such that
\[
\lambda (\pi (u) v) = \theta (u) \lambda (v) \text{ for all } u \in U, v \in V. 
\]
Such a linear form is called a Whittaker functional, and the space of those
has dimension 1 (if they exist).
Let $W(\pi,\theta)$ be the space of all functions $W : G \to \C$ of the form
\[
W_v (g) = \lambda (\pi (g) v) \qquad g \in G, v \in V.
\]
Then $W(\pi,\theta)$ is stable under right translations and the representation
thus obtained is isomorphic to $\pi$ via $v \leftrightarrow W_v$. 
Most irreducible representations of $\GL_n (F)$,
and in particular all the supercupidal ones, are generic \cite{GeKa}. 

We consider one irreducible generic representation $\pi$ of $\GL_n (F)$ and 
another one, $\pi'$, of $\GL_{n-1}(F)$.
For $W \in W(\pi,\theta)$ and $W' \in W(\pi',\theta)$ one defines the integral
\begin{equation}\label{eq:A.11}
\Psi (s,W,W') = \int_{U_{n-1} \backslash \GL_{n-1}(F)} W \matje{g}{0}{0}{1}
\overline{W'(g)}\, | \det (g) |_F^{s-1/2} \textup{d}\mu (g) ,
\end{equation}
where $\mu$ denotes the quotient of Haar measures on $\GL_{n-1}(F)$ and on $U_{n-1}$.
This integral is known to converge absolutely when Re$(s)$ is large 
\cite[Theorem 2.7.i]{JPS2}. The contragredient representations $\check \pi$ and 
$\check \pi'$ are also generic. We define 
$\check W \in W(\check \pi,\overline{\theta})$ by
\[
\check W (g) = W (w_n g^{-T}) \qquad g \in \GL_n (F) , 
\]
where $g^{-T}$ is the transpose inverse of $g$ and $w_n$ is the permutation
matrix with ones on the diagonal from the lower left to the upper right corner.

We denote the central character of $\pi'$ by $\omega_{\pi'}$.
With these notations the $L$-functions, $\vareps$-factors and $\gamma$-factors 
of the pair $(\pi,\pi')$ are related by
\begin{align}
\frac{\Psi (s,W,W')}{L(s,\pi \times \pi')} \vareps (s,\pi \times \pi',\psi) = 
\omega_{\pi'}(-1)^{n-1} \frac{\Psi (1-s,\check W,\check W')}{
L(1-s,\check \pi \times \check \pi')} , \label{eq:A.12} \\
\gamma (s,\pi \times \pi',\psi) = \vareps (s,\pi \times \pi',\psi) 
\frac{L(1-s,\check \pi \times \check \pi')}{L(s,\pi \times \pi')} , \label{eq:A.13}
\end{align}
see \cite[Theorem 2.7.iii]{JPS2}.
We regard these equations as definitions of the $\vareps$- and $\gamma$-factors.

\begin{thm} \label{thm:LLCr}
Let $\pi$ be a supercuspidal representation in $\Irr(\GL_n(F),K_{r,n})$, 
with $r\in\Z_{>0}$. Let $\phi \in \Phi (\GL_n(F))$ be an elliptic 
parameter such that $\SL_2(\C)\subset\ker(\phi)$ and $d(\phi) \leq r-1$. Suppose that
$\det \phi$ corresponds to the central character of $\pi$ via local class field theory
and that
\[
\vareps(s,\pi\times \pi',\psi) = \vareps(s,\phi\otimes\rec_{F,n'}(\pi'),\psi)
\]
holds in one of the following cases:
\enuma{
\item for $n' = n-1$ and every generic $\pi' \in \Irr(\GL_{n'}(F),K_{2r-1,n'})$;
\item for every $n'$ such that $1 \leq n' < n$, and for every supercuspidal representation 
$\pi'$ in $\Irr(\GL_{n'}(F),K_{2r-1,n'})$.
}
Then $\phi=\rec_{F,n}(\pi)$.
\end{thm} 
\begin{proof} 
The claim is trivially true if $n=1$, so we assume $n \geq 2$ throughout the proof.
We would like to show that 
\begin{equation}\label{eq:14}
\vareps(s,\pi\times \pi',\psi) = \vareps(s,\rec_{F,n}^{-1}(\phi) \otimes \pi',\psi)
\end{equation}
for every generic representation $\pi'$ of $\GL_{n-1}(F)$, for then we are in the
setting of \cite{Hen1}. 
Since $\pi$ and $\rec_{F,n}^{-1}(\phi)$ are supercuspidal
\begin{equation} \label{eqn:Ls}
L(s,\pi \times \sigma ) = L(s,\rec_{F,n}^{-1}(\phi) \times \sigma ) =
L(s,\check\pi \times \check \sigma) = L(s,\rec_{F,n}^{-1}(\check \phi) \times \check \sigma ) = 1
\end{equation}
for any generic representation $\sigma$ of a general linear group of smaller size 
\cite[Theorem 8.1]{JPS2}. In view of \eqref{eq:A.13} we might just as well check 
\eqref{eq:A.14} with $\gamma$-factors instead of $\vareps$-factors. We proceed as in the 
proof of \cite[(3.3.4)]{Hen1}. Using the multiplicativity of $\gamma$-factors 
\cite[Theorem 3.1]{JPS2}, we first write $\gamma(s,\pi\times \pi',\psi)$ as a product of 
$\gamma(s,\pi\times\langle\Delta_i\rangle^t,\psi)$, where $\Delta_i$ 
is a Zelevinsky segment. Next we write $\gamma(s,\pi \times \langle \Delta_i\rangle^t,\psi)$ 
itself as a product 
\[
\prod \nolimits_{h} \gamma(s,\pi \times \pi_i' |\;|^h ,\psi),
\]
with $\pi'_i$ supercuspidal. The same argument yields the equality
\[
\gamma (s,\rec_{F,n}^{-1}(\phi) \otimes \pi',\psi) = 
\prod\nolimits_{i,h} \gamma(s,\rec_{F,n}^{-1}(\phi) \times \pi_i' |\;|^h ,\psi) .
\]
Hence, to establish \eqref{eq:14} it suffices to show that
\begin{equation}\label{eq:15}
\gamma(s,\pi \times \sigma,\psi) = \gamma (s,\rec_{F,n}^{-1}(\phi) \times \sigma', \psi)
\end{equation}
whenever $\sigma' \in \Irr (\GL_{n'}(F))$ is a supercuspidal and $1 \leq n' < n$. 
By Proposition~\ref{lem:pi} $d(\rec_{F,n}^{-1}(\phi))=d(\phi) \leq r-1$,
and by Lemma~\ref{lem:depthKr}  $d(\pi)\le r-1$, so in particular
\[
2 \max \big\{ d(\pi), d(\rec_{F,n}^{-1}(\phi)) \big\} \leq 2 r - 2.
\]
This and the condition on the central character of $\pi$ mean that all the assumptions
used in the proof 
of \cite[Theorem 8.4]{Gan2} are fulfilled if $d(\sigma') > 2r - 2$. The result 
\cite[Theorem 8.4]{Gan2}, which uses Proposition \ref{lem:pi}, says that 
\eqref{eq:15} holds if $d(\sigma') > 2r - 2$. 

Next we consider a supercuspidal $\sigma' \in \GL_{n'}(F)$ of depth $\leq 2r-2$. 
By Lemma \ref{lem:depthKr} the representation $\sigma'$ has nonzero $K_{2r-1,n'}$-fixed vectors, so \eqref{eq:15}
holds under the assumption (b). If we are assuming (a) instead, consider  
\[
\pi' := I_{\GL_{n'}(F) \times \GL_1 (F)^{n - 1 - n'}}^{\GL_{n-1} (F)} (\sigma' \otimes 
\mathrm{triv}_{\GL_1 (F)}^{\otimes n - 1 - n'}) .
\]
This representation lies in $\Irr (\GL_{n-1}(F),K_{2r-1,n-1})$ and is generic because it
is irreducibly induced from a supercuspidal representation, so by the assumption of
the theorem
\begin{equation}\label{eq:16}
\gamma (s,\pi\times \pi',\psi) = \gamma (s,\rec_{n,F}^{-1} \phi \times \pi',\psi) .
\end{equation}
By \cite[Theorem 3.1]{JPS2}
\begin{multline}\label{eq:17}
\gamma (s,\pi \times \pi',\psi) = \gamma (s,\pi \times \sigma, \psi) 
\prod\nolimits_{j=1}^{n-1-n'} \gamma (s,\pi \times \mathrm{triv}_{\GL_1 (F)} ,\psi) \\
= \gamma (s,\pi \times \sigma', \psi) \gamma (s,\pi , \psi )^{n - 1 - n'} ,
\end{multline}
and the same holds with $\rec_{F,n}^{-1}(\phi)$ instead of $\pi$.
To deal with $\gamma (s,\pi , \psi )$ (only if $n > 2$, otherwise $n-1-n' = 0$) 
we take a look at a representation of the form 
\[
\tau : = I_{\GL_1 (F) \times \GL_{n-2}(F)}^{\GL_{n-1}(F)} 
\big( \mathrm{triv}_{\GL_1 (F)} \otimes \sigma \big) ,
\]
where $\sigma \in \Irr (\GL_{n-2}(F))$ is supercuspidal and $d(\sigma) > 2r - 2$. 
By the above
\begin{multline}\label{eq:18}
\gamma (s,\pi , \psi ) = \gamma (s,\pi \times \tau,\psi) 
\gamma (s,\pi \times \sigma,\psi )^{-1} \\
= \gamma (s,\rec_{F,n}^{-1}(\phi) \times \tau,\psi) \gamma (s,\rec_{F,n}^{-1}(\phi) 
\times \sigma,\psi )^{-1} = \gamma (s,\rec_{F,n}^{-1}(\phi),\psi).
\end{multline}
It follows from \eqref{eq:16}, \eqref{eq:17} and \eqref{eq:18} that
\begin{multline*}
\gamma (s,\pi \times \sigma', \psi) = \gamma (s,\pi \times \pi',\psi) 
\gamma (s,\pi , \psi )^{n' + 1 - n} \\
\gamma (s,\rec_{n,F}^{-1} \phi \times \pi',\psi) \gamma (s,\rec_{F,n}^{-1}(\phi),
\psi)^{n' + 1 - n} = \gamma (s,\rec_{n,F}^{-1} \phi \times \sigma,\psi) .
\end{multline*}
This finishes the proof of \eqref{eq:15} and of \eqref{eq:14}. Now we can apply
\cite[Th\'eor\`eme 1.1]{Hen1}, which says that $\pi \cong \rec_{F,n}^{-1}(\phi)$.
\end{proof}

\section{The method of close fields}
\label{sec:method}

Kazhdan's method of close fields \cite{Kaz,Del} has proven useful to generalize results 
that are known for groups over $p$-adic fields to groups over local fields of positive 
characteristic. It was worked out for inner forms of $\GL_n (F)$ by Badulescu \cite{Bad1}.

Let $F$ and $\widetilde F$ be two local non-archimedean fields which are close.
Let $G = \GL_m (D)$ be a standard inner form of $\GL_n (F)$ and let $\widetilde{G} = 
\GL_m (\widetilde D)$ be the standard inner form of $\GL_n (\widetilde F)$ with the 
same Hasse invariant as $G$. 

In this section, an object with a tilde will always be the counterpart over
$\widetilde F$ of an object (without tilde) over $F$, and a superscript $\sharp$ means
the subgroup of elements with reduced norm 1. Then
$\widetilde{G}^\sharp = \widetilde{G}_{\der}$ is an inner form of $\SL_n (\widetilde F)$
with the same Hasse invariant as $G^\sharp$ and 
\[
\chi_{\widetilde{G}} = \chi_{\widetilde{G}^\sharp} = \chi_{G^\sharp} = \chi_G .
\]
Let $\mf o_D$ be the ring of integers of $D ,\; \varpi_D$ a uniformizer and 
$\mf p_D = \varpi_D \mf o_D$ its unique maximal ideal. The explicit multiplication rules
in $D$ \cite[Proposition IX.4.11]{Wei2} show that we may assume that a power 
of $\varpi_D$ equals $\varpi_F$, a uniformizer of $F$.

Generalizing the notation for $\GL_n (F)$, let $K_0 = \GL_m (\mf o_D)$ 
be the standard maximal compact subgroup of $G$ and define, for $r \in \Z_{>0}$:
\[
K_r = \ker \big( \GL_m (\mf o_D) \to \GL_m (\mf o_D / \mf p_D^r) \big) =
1 + \Mat_m (\mf p_D^r) .
\]
We denote the category of smooth $G$-representations that are generated
by their $K_r$-invariant vectors by $\Mod (G,K_r)$. Let $\mc H (G,K_r)$ be the 
convolution algebra of compactly supported $K_r$-biinvariant functions $G \to \C$.
According to \cite[Corollaire 3.9]{BeDe}
\begin{equation}\label{eq:A.3}
\begin{array}{ccc}
\Mod (G,K_r) & \to & \Mod (\mc H (G,K_r)) , \\
V & \mapsto & V^{K_r}  
\end{array}
\end{equation}
is an equivalence of categories. The same holds for $(\widetilde{G},\widetilde K_r)$.

From now on we suppose that $F$ and $\widetilde F$ are $l$-close for some $l \geq r$,
that is, 
\begin{equation}\label{eq:A.2}
\mf o_F / \mf p_F^l \cong \mf o_{\widetilde F} / \mf p_{\widetilde F}^l 
\end{equation}
as rings. As remarked in \cite{Del}, for every local field of characteristic $p > 0$
and every $l \in \N$ there exists a finite extension of $\Q_p$ which is $l$-close to $F$.

Notice that \eqref{eq:A.2} induces a group isomorphism $\mf o_F^\times / 1 + \mf p_F^l
\cong \mf o_{\widetilde F}^\times / 1 + \mf p_{\widetilde F}^l$. A choice of uniformizers
$\varpi_F$ and $\varpi_{\widetilde F}$ then leads to
\begin{equation}\label{eq:A.7}
F^\times / 1 + \mf p_F^l \cong \Z \times \mf o_F^\times / 1 + \mf p_F^l \cong
\Z \times \mf o_{\widetilde F}^l / 1 + \mf p_{\widetilde F}^l \cong 
\widetilde{F}^\times / 1 + \mf p_{\widetilde F}^l .
\end{equation}
With \cite[Th\'eor\`eme 2.4]{Bad1}, \eqref{eq:A.2} also gives rise to a ring isomorphism
\begin{equation}\label{eq:A.14}
\lambda_r : \mf o_D  / \mf p_D^r \to \mf o_{\widetilde D} / \mf p_{\widetilde D}^r , 
\end{equation}
which in turn induces a group isomorphism
\[
\GL_m (\lambda_r) : K_0 / K_r = \GL_m ( \mf o_D  / \mf p_D^r) \to 
\widetilde K_0 / \widetilde K_r = \GL_m ( \mf o_{\widetilde D} / \mf p_{\widetilde D}^r ) .
\]
Recall that the Cartan decomposition for $G$ says that $K_0 \backslash G / K_0$ can
be represented by
\[
A^+ := \{ \mathrm{diag}(\varpi_D^{a_1}, \ldots, \varpi_D^{a_m}) \in \GL_m (D) : 
a_1 \leq \ldots \leq a_m \} . 
\]
Clearly $A^+$ is canonically in bijection with the analogous set $\widetilde A^+$ 
of representatives for $\widetilde K_0 \backslash \widetilde G / \widetilde K_0$
(which of course depends on the choice of a uniformizer $\varpi_{\widetilde D}$).
Since $K_r \backslash G / K_r$ can be identified with $K_r \backslash K_0 \times 
A^+ \times K_0 / K_r$, that and $\GL_m (\lambda_r)$ determine a bijection
\begin{equation}\label{eq:A.1}
\zeta_r : K_r \backslash G / K_r \to
\widetilde K_r \backslash \widetilde G / \widetilde K_r .
\end{equation}
Most of the next result can be found in \cite{Bad1,BHLS}.

\begin{thm}\label{thm:A.1}
Suppose that $F$ and $\widetilde F$ are sufficiently close, in the sense that the
$l$ in \eqref{eq:A.2} is large. Then the map $1_{K_r g K_r} \mapsto 
1_{\zeta_r (K_r g K_r)}$ extends to a $\C$-algebra isomorphism
\[
\zeta_r^G : \mc H (G,K_r) \to \mc H (\widetilde G,\widetilde K_r) . 
\]
This induces an equivalence of categories
\[
\overline{\zeta_r^G} : \Mod (G,K_r) \to \Mod (\widetilde G,\widetilde K_r)
\]
such that:
\enuma{
\item $\overline{\zeta_r^G}$ respects twists by unramified characters and its effect
on central characters is that of \eqref{eq:A.7}.
\item For irreducible representations, $\overline{\zeta_r^G}$ preserves temperedness, 
essential square-integrability and cuspidality.
\item Let be $P$ a parabolic subgroup of $G$ with a Levi factor $M$ which is standard, and
let $\widetilde P$ and $\widetilde M$ be the corresponding subgroups of $\widetilde G$. Then
\[
\begin{array}{ccc}
\Mod (G,K_r) & \xrightarrow{\; \overline{\zeta_r^G} \;} & 
\Mod (\widetilde G,\widetilde K_r) \\
\uparrow I_P^G & & \uparrow I_{\widetilde P}^{\widetilde G} \\
\Mod (M,K_r \cap M) & \xrightarrow{\; \overline{\zeta_r^M} \;} & 
\Mod (\widetilde M,\widetilde K_r \cap \widetilde M)
\end{array} 
\]
commutes.
\item $\overline{\zeta_r^G}$ commutes with the formation of contragredient representations.
\item $\overline{\zeta_r^G}$ preserves the L-functions, $\vareps$-factors and
$\gamma$-factors.
} 
\end{thm}
\begin{proof}
The existence of the isomorphism $\zeta_r^G$ is \cite[Th\'eor\`eme 2.13]{Bad1}. 
The equivalence of categories follows from that and \eqref{eq:A.3}. \\
(a) Let $G^1$ be the subgroup of $G$ generated by all compact subgroups of $G$, that is,
the intersection of the kernels of all unramified characters of $G$. Since $K_r$ and
$\widetilde K_r$ are compact, $\zeta_r$ restricts to a bijection 
\[
K_r \backslash G^1 / K_r \to \widetilde K_r \backslash \widetilde G^1 / \widetilde K_r .
\]
Moreover, because $A^+ \to \widetilde A^+$ respects the group multiplication whenever 
it is defined, the induced bijection $G / G^1 \to \widetilde G / \widetilde G^1$ is in
fact a group isomorphism. Hence $\zeta_r$ induces an isomorphism
\[
\overline{\zeta_r^{G / G^1}} : X_{nr}(G) = \Irr (G / G^1) \to 
\Irr (\widetilde G / \widetilde G^1) = X_{nr}(\widetilde G) ,
\]
which clearly satisfies, for $\pi \in \Mod (G,K_r)$ and $\chi \in X_{nr}(G)$: 
\[
\overline{\zeta_r^G} (\pi \otimes \chi) = \overline{\zeta_r^G} (\pi) 
\otimes \overline{\zeta_r^{G / G^1}} (\chi) .
\]
The central characters can be dealt with similarly. The characters of $Z(G)$ appearing
in Mod$(G,K_r)$ are those of 
\[
Z(G) / Z(G) \cap K_r = F^\times / 1 + \mf p_F^r . 
\]
Now we note that $\zeta_r^G$ and \eqref{eq:A.7} have the same restriction to 
the above group. \\
(b) By \cite[Th\'eor\`eme 2.17]{Bad1}, $\overline{\zeta_r^G}$ preserves cuspidality
and square-integrability modulo centre. Combining the latter with part (a), we find
that it also preserves essential square-integrability. A variation on the proof of
\cite[Th\'eor\`eme 2.17.b]{Bad1} shows that temperedness is preserved as well.
Alternatively, one can note that every irreducible tempered representation in
$\Mod (G,K_r)$ is obtained with parabolic induction from a square-integrable modulo
centre representation in $\Mod (M,M \cap K_r)$, and then apply part (c).\\
(c) This property, and its analogue for Jacquet restriction, are proven in 
\cite[Proposition 3.15]{BHLS}. We prefer a more direct argument.
The constructions in \cite[\S 2]{Bad1} apply equally well to $(M,K_r \cap M)$,
so $\zeta_r$ induces an algebra isomorphism $\zeta_r^M$ and an equivalence of
categories $\overline{\zeta_r^M}$. By \cite[Corollary 7.12]{BuKu} the parabolic
subgroup $P$ determines an injective algebra homomorphism
\[
t_P : \mc H (M,K_r \cap M) \to \mc H (G,K_r) . 
\]
This in turn gives a functor
\[
\begin{array}{cccc}
(t_P)_* : & \Mod (\mc H (M,K_r \cap M)) & \to & \Mod (\mc H (G,K_r)) , \\ 
& V & \mapsto & \Hom_{\mc H (M,K_r \cap M)}(\mc H (G,K_r),V) ,  
\end{array}
\]
where $\mc H (G,K_r)$ and $V$ are regarded as $\mc H (M,K_r \cap M)$-modules
via $t_P$. This is a counterpart of parabolic induction, in the sense that
\begin{equation}\label{eq:A.4}
\begin{array}{ccc}
\Mod (G,K_r) & \to & 
\Mod (\mc H (G,K_r)) \\
\uparrow I_P^G & & \uparrow (t_P)_* \\
\Mod (M,K_r \cap M) & \to & \Mod (\mc H (M,K_r \cap M))
\end{array} 
\end{equation}
commutes \cite[Corollary 8.4]{BuKu}.
The construction of $t_P$ in \cite[\S 7]{BuKu} depends only on properties that 
are preserved by $\zeta_r^G$ (and its counterparts for other groups), so 
\begin{equation}\label{eq:A.5}
\begin{array}{ccc}
\mc H (G,K_r) & \to & 
\mc H (\widetilde G,\widetilde K_r) \\
\uparrow (t_P)_*  & & \uparrow (t_{\widetilde P})_* \\
\mc H (M,K_r \cap M) & \to & \mc H (\widetilde M,\widetilde K_r \cap \widetilde M)
\end{array}  
\end{equation}
commutes. Now we combine \eqref{eq:A.5} with \eqref{eq:A.4} for $G$ and $\widetilde G$.\\
(d) The contragredient of a $\mc H (G,K_r)$-module is defined via the involution
$f^* (g) = f (g^{-1})$.
The equivalence of categories \eqref{eq:A.3} commutes with the formation of 
contragredients because $(V^*)^{K_r} \cong (V^{K_r})^*$. The map $\overline{\zeta_r^G}$
does so because $\zeta_r^G$ commutes with the involution *. \\
(e) For the $\gamma$-factors see \cite[Th\'eor\`eme 2.19]{Bad1}.

Consider the L-function of a supercuspidal $\sigma \in \Irr (G,K_r)$.
By \cite[Propositions 4.4 and 5.11]{GoJa} $L(s,\sigma) = 1$ unless $m = 1$ and $\sigma = 
\chi \circ $Nrd with $\chi : F^\times \to \C^\times$ unramified. This property is preserved 
by $\zeta_r^G$, so $L \big( s,\overline{\zeta_r^G}(\sigma) \big) = 1$ if the condition is 
fulfilled. In the remaining case 
\[
L(s,\sigma) = L(s + (d-1)/2,\chi) =  \big( 1 - q^{-s + (1-d)/2} \chi (\varpi_F) \big)^{-1} .
\] 
The proof of part (a) shows that $\overline{\zeta_r^G}(\sigma) = 
\chi \circ \zeta_r^{F^\times} \circ$ Nrd, so
\[
L \big( s,\overline{\zeta_r^G}(\sigma) \big) = 
\big( 1 - q^{-s + (1-d)/2} \chi (\zeta_r^{F^\times} (\varpi_{\widetilde F})) \big)^{-1} =
\big( 1 - q^{-s + (1-d)/2} \chi (\varpi_F) \big)^{-1} .
\]
Thus $\overline{\zeta_r^G}$ preserve the L-functions of supercuspidal representations.
By \cite[\S 3]{Jac} the L-functions of general $\pi \in \Irr (G,K_r)$ are determined by the
L-functions of supercuspidal representations of Levi subgroups of $G$, in combination
with parabolic induction and twisting with unramified characters. In view of parts (a),(c)
and the above, this implies that $\overline{\zeta_r^G}$ always preserves L-functions.

Now the relation
\[
\epsilon (s,\pi,\psi) = \gamma (s,\pi,\psi) \frac{L(s,\pi)}{L(1-s,\check \pi)} 
\]
and part (d) show that $\overline{\zeta_r^G}$ preserves $\epsilon$-factors.
\end{proof}

In \cite{Bad3} Badulescu showed that Theorem \ref{thm:A.1} has an analogue for $G^\sharp$ 
and $\widetilde{G}^\sharp$, which can easily be deduced from Theorem \ref{thm:A.1}.
We quickly recall how this works. Note that $M$ is a central extension of 
$M^\sharp = \{ m \in M : \mathrm{Nrd}(m) = 1\}$. 
A few properties of the reduced norm \cite[\S IX.2 and equation IX.4.9]{Wei2} entail 
\begin{equation}\label{eq:A.6}
\begin{split}
& \mr{Nrd}(K_r \cap M) = \mr{Nrd}(1 + \mf p_D^r) = 1 + \mf p_F^r , \\
& M^\sharp (K_r \cap M) = \{ m \in M : \mr{Nrd}(m) \in 1 + \mf p_F^r \}.  
\end{split}
\end{equation}
Choose the Haar measures on $M$ and $M^\sharp$ so that vol$(K_r \cap M) = 
\mr{vol}(K_r \cap M^\sharp)$. The inclusion $M^\sharp \to M$ induces an algebra isomorphism
\begin{multline*}
\mc H (M^\sharp ,K_r \cap M^\sharp) \to \mc H (M^\sharp (K_r \cap M),K_r \cap M) \\ :=
\{ f \in \mc H(M,K_r \cap M) : \mr{supp}(f) \subset M^\sharp (K_r \cap M) \} . 
\end{multline*}
In view of \eqref{eq:A.6} and the isomorphism $\mf o_F / \mf p_F^r \cong
\mf o_{\widetilde F} / \mf p_{\widetilde F}^r ,\; \zeta_r^M$ yields a bijection
\[
\mc H (M^\sharp (K_r \cap M),K_r \cap M) \to \mc H (\widetilde M^\sharp 
(\widetilde K_r \cap \widetilde M),\widetilde K_r \cap \widetilde M). 
\]
Hence it induces an algebra isomorphism
\[
\zeta_r^{M^\sharp} : \mc H (M^\sharp,K_r \cap M^\sharp) \to 
\mc H (\widetilde M^\sharp,\widetilde K_r \cap \widetilde M^\sharp) .
\]
\begin{cor}\label{cor:A.2}
Theorem \ref{thm:A.1} (except part e) also holds for the corresponding subgroups 
of elements with reduced norm 1. 
\end{cor}
\begin{proof}
Using the isomorphisms $\zeta_r^{M^\sharp}$, this can be proven in the same way as
Theorem \ref{thm:A.1} itself.
\end{proof}

As preparation for the next section, we will show that in certain special cases
the functors $\overline{\zeta_r^G}$ preserve the L-functions, $\vareps$-factors
and $\gamma$-factors of pairs of representations, as defined in \cite{JPS2}.

Suppose that $\widetilde{F}$ is $l$-close to $F$ and that 
$\widetilde{\psi} : \widetilde{F} \to \C^\times$ is a character which is trivial
on $\mf o_{\widetilde F}$. We say that $\widetilde{\psi}$ is $l$-close to $\psi$ if 
$\widetilde{\psi} \big|_{\varpi_{\widetilde F}^{-l} \mf o_{\widetilde F} / \mf o_{\widetilde F}}$ 
corresponds to $\psi \big|_{\varpi_F^{-l} \mf o_F / \mf o_F}$ under the isomorphisms
\[
\varpi_{\widetilde F}^{-l} \mf o_{\widetilde F} / \mf o_{\widetilde F} \cong 
\mf o_{\widetilde F} / \varpi_{\widetilde F}^l \mf o_{\widetilde F} \cong
\mf o_F / \varpi_F^l \mf o_F \cong \varpi_F^{-l} \mf o_F / \mf o_F .
\]

\begin{thm}\label{thm:A.3}
Assume that $F$ and $\widetilde{F}$ are $l$-close for some $l > r$ and that
$\widetilde{\psi}$ is $l$-close to $\psi$. Let $\pi \in \Irr (\GL_n (F),K_{r,n})$ be 
supercuspidal and let $\pi' \in \Irr (\GL_{n-1}(F),K_{r,n-1})$ be generic. Then
\begin{align*}
& L \big( s, \overline{\zeta_r^{\GL_n (F)}}(\pi) \times 
\overline{\zeta_r^{\GL_{n-1}(F)}}(\pi') \big) \; = \; L (s,\pi \times \pi') \; = \; 1, \\
& \vareps \big( s, \overline{\zeta_r^{\GL_n (F)}}(\pi) \times 
\overline{\zeta_r^{\GL_{n-1}(F)}}
(\pi'), \widetilde{\psi} \big) \; = \; \vareps (s,\pi \times \pi',\psi) ,\\
& \gamma \big( s, \overline{\zeta_r^{\GL_n (F)}}(\pi) \times \overline
{\zeta_r^{\GL_{n-1}(F)}}
(\pi'), \widetilde{\psi} \big) \; = \; \gamma (s,\pi \times \pi',\psi) .
\end{align*}
\end{thm}

\begin{proof}
Since $\pi$ and $\check \pi$ are supercuspidal, whereas $\pi'$ and $\check \pi'$ are 
representations of a general linear group of lower rank, \cite[Theorem 8.1]{JPS2} assures that
all the L-functions appearing here are 1. By \eqref{eq:A.13} this implies that the relevant
$\gamma$-factors are equal to the $\vareps$-factors of the same pairs. Hence it suffices to
prove the claim for the $\vareps$-factors. 
We note that by Theorem \ref{thm:A.1} 
\begin{equation}\label{eq:A.20}
\omega_{\pi'}(-1)^{n-1} = 
\omega_{\overline{\zeta_r^{\GL_{n-1}(F)}}\pi'} (-1)^{n-1} ,
\end{equation}
so from \eqref{eq:A.12} we see that it boils down to comparing the integrals $\Psi (s,W,W')$ 
and $\Psi (1-s,\check W,\check W')$ with their versions for $\widetilde F$.

Fix a Whittaker functional $\lambda'$ for $(\pi',V')$ and a vector $v' \in V^{K_{r,n-1}}$. 
Then $W' := W_{v'} \in W (\pi',\theta)$ is right $K_{r,n-1}$-invariant. Similarly we pick 
$W = W_v \in W(\pi,\theta)$, but now we have to require only that $W$ is right invariant under 
$K_{r,n-1}$ on $\GL_{n-1}(F) \subset \GL_n (F)$. Because $\theta$ is unitary, the function
\[
\GL_{n-1}(F) \to \C \colon g \mapsto W\matje{g}{0}{0}{1} \overline{W'(g)}
\]
is constant on sets of the form $U_{n-1} g K_{r,n-1}$. Since the subgroup 
$K_{r,n-1}$ is stable under the automorphism $g \mapsto g^{-T}$, the functions 
$\check W$ and $\check W'$ are also right $K_{r,n-1}$-invariant. Both transform under 
left translations by $U_{n-1}$ as $\overline{\theta}$, so
\[
\GL_{n-1}(F) \to \C \colon g \mapsto
\check W \matje{g}{0}{0}{1} \overline{\check W'(g)}
\]
defines a function $U_{n-1} \backslash \GL_{n-1}(F) / K_{r,n-1} \to \C$. Since 
$\det (K_{r,n-1}) \subset \mf o_F^\times$ and $\det (U_{n-1}) = 1 $, 
the function $| \det |_F$ can also be regarded as a map \\
$U_{n-1} \backslash \GL_{n-1}(F) / K_{r,n-1} \to \C$.

Now the idea is to transfer these functions to objects over $\widetilde{F}$ by means 
of the Iwasawa decomposition as in \cite[\S 3]{Lem}, and to show that neither side of
\eqref{eq:A.12} changes.

Let $A_{\varpi_F} \subset \GL_{n'}(F)$ be the group of diagonal matrices all whose 
entries are powers of $\varpi_F$. The Iwasawa decomposition states that
\begin{equation}\label{eq:A.15}
\GL_{n}(F) = \bigsqcup\nolimits_{a \in A_{\varpi_F}} U_{n} a K_{0,n} .
\end{equation}
This, the canonical bijection $A_{\varpi_F} \to A_{\varpi_{\widetilde F}} 
\colon a \mapsto \widetilde{a}$ 
and the isomorphism $\GL_{n}(\lambda_r)$ from \eqref{eq:A.14} combine to a bijection
\begin{equation}\label{eq:A.16}
\begin{aligned}
\zeta'_r \colon U_{n} \backslash & \GL_{n}(F) / K_{r,n} & \to & \;
\widetilde{U}_{n} \backslash \GL_{n}(\widetilde{F}) / \widetilde{K}_{r,n} , \\
& U_{n} \,a k\, K_{r,n} & \mapsto &
\; \widetilde{U}_{n} \,\widetilde{a} \GL_{n}(\lambda_r)(k) \, \widetilde{K}_{r,n} .
\end{aligned}
\end{equation}
Because $\widetilde \psi$ is $l$-close to $\psi$ we may apply \cite[Lemme 3.2.1]{Lem}, which
says that there is a unique linear bijection 
\begin{equation}\label{eq:A.8}
\rho_{n} \colon W(\pi,\theta)^{K_{r,n}} \to 
W \big( \overline{\zeta_r^{\GL_{n}(F)}}(\pi),\widetilde{\theta} \big)^{\widetilde{K}_{r,n}}
\end{equation}
which transforms the restriction of functions to $A_{\varpi_F} K_{0,n}$ 
according to $\zeta'_r$. We will use \eqref{eq:A.16} and \eqref{eq:A.8} also with $n-1$ instead of $n$.

Put $\widetilde{W} = \rho_n(W)$ and $\widetilde{W'} = \rho_{n-1} (W')$. As \eqref{eq:A.16}
commutes with $g \mapsto g^{-T}$, 
\begin{equation}\label{eq:A.17}
\check{\widetilde{W}} = \rho_n (\check{W}) \text{ and } 
\check{\widetilde{W'}} = \rho_{n-1} (\check{W'}) .
\end{equation}
These constructions entail that 
\[
\GL_{n-1}(\widetilde{F}) \to \C \colon \widetilde{g} \mapsto 
\widetilde{W} \matje{\widetilde{g}}{0}{0}{1} \overline{\widetilde{W}'(\widetilde{g})}
\]
defines a function $\widetilde{U}_{n-1} \backslash \GL_{n-1}(\widetilde{F}) / 
\widetilde{K}_{r,n-1} \to \C$, and that
\begin{equation}\label{eq:A.18}
W \matje{g}{0}{0}{1} \overline{W' (g)}
= \widetilde{W} \matje{\zeta'_r (g)}{0}{0}{1} \overline{\widetilde{W}'(\zeta'_r (g))} .
\end{equation}
It follows immediately from the definition of $\zeta'_r$ that 
\begin{equation}\label{eq:A.19}
| \det (\zeta'_r (g)) |_{\widetilde F} = | \det (g) |_F .
\end{equation}
For the computation of $\Psi (s,W,W')$ we may normalize the measure $\mu$ such that 
every double coset $U_{n-1} \backslash U_{n-1} g K_{r,n-1}$ has volume $1$, and similarly 
for the measure on $\widetilde{U}_{n-1} \backslash \GL_{n-1}(\widetilde F)$.
The equalities \eqref{eq:A.18} and \eqref{eq:A.19} imply
\begin{multline*}
\Psi (s,W,W') = \sum_{g \in A_{\varpi_F} K_{0,n-1} / K_{r,n-1}} 
W \matje{g}{0}{0}{1} \overline{W'(g)} | \det (g) |_F^{s - 1/2} \\
= \sum_{\widetilde{g} \in A_{\varpi_{\widetilde F}}  
\widetilde{K}_{0,n-1} / \widetilde{K}_{r,n-1}} \widetilde{W} 
\matje{\widetilde{g}}{0}{0}{1} 
\overline{\widetilde{W}'(\widetilde{g})} | \det (\widetilde{g}) 
|_{\widetilde F}^{s - 1/2} = 
\Psi (s,\widetilde{W},\widetilde{W}').
\end{multline*}
An analogous computation, additionally using \eqref{eq:A.17}, shows that 
\[
\Psi (s,\check W,\check W') = \Psi( s,\check{\widetilde{W}}, \check{\widetilde{W}'}) .
\]
The previous two equalities and \eqref{eq:A.20} prove that all terms in \eqref{eq:A.12},
expect possibly the $\vareps$-factors, have the same values as the corresponding
terms defined over $\widetilde F$. To establish the desired equality of $\vareps$-factors,
it remains to check that $\Psi (s,W,W')$ is nonzero for a suitable choice of 
right $K_{r,n-1}$-invariant functions $W$ and $W'$. 

Take $v'$ as above, but nonzero. Then $W' = W_{v'}$ is nonzero because $V' \cong W(\pi',\theta)$. 
Choose $g_0 \in \GL_{n-1}(F)$ with $W'(g_0) \neq 0$ and define $H : \GL_{n-1}(F) \to \C$ by
$H(g) = W'(g)$ if $g \in U_{n-1} g_0 K_{r,n-1}$ and $H(g) = 0$ otherwise.
According to \cite[Lemme 2.4.1]{Hen1}, there exists $W \in W(\pi,\psi)$ such that 
$W \matje{g}{0}{0}{1} = H(g)$ for all $g \in \GL_{n-1}(F)$. Notice that such a $W$ is 
automatically right invariant under $K_{r,n-1}$ on $\GL_{n-1}(F) \subset \GL_n (F)$. 
Now we can easily compute the required integral:
\begin{align*}
\Psi (s,W,W') & = \int_{U_{n-1} \backslash \GL_{n-1}(F)} |H(g)|^2 |\det (g)|_F^{s - 1/2} 
\textup{d} \mu (g) \\ 
& = \int_{U_{n-1} \backslash U_{n-1} g_0 K_{r,n-1}} |W' (g)|^2 
|\det (g)|_F^{s - 1/2} \textup{d} \mu (g) \\
& = \mu (U_{n-1} \backslash U_{n-1} g_0 K_{r,n-1}) |W'(g_0)|^2 |\det (g_0) |_F^{s - 1/2} 
\neq 0 . \qedhere
\end{align*}
\end{proof}

\section{Close fields and Langlands parameters}
\label{sec:CL}

The section is based on Deligne's comparison of the Galois groups of close fields. 
According to \cite[(3.5.1)]{Del} the isomorphism
\eqref{eq:A.2} gives rise to an isomorphism of profinite groups
\begin{equation}\label{eq:8.8}
\text{Gal}(F_s / F) / \text{Gal}(F_s / F)^l \cong \text{Gal}(\widetilde{F}_s / 
\widetilde{F}) / \text{Gal}(\widetilde{F}_s / \widetilde{F})^l ,
\end{equation}
which is unique up to inner automorphisms. Since both $\mathbf{W}_F$ and 
$\mathbf{W}_{\widetilde F}$ can be described in terms of automorphisms of the residue 
field $\mf o_F / \mf p_F \cong \mf o_{\widetilde F} / \mf p_{\widetilde F}$, 
\eqref{eq:8.8} restricts to an isomorphism
\begin{equation}\label{eq:8.9}
\mathbf{W}_F / \text{Gal}(F_s / F)^l \cong 
\mathbf{W}_{\widetilde F} / \text{Gal}(\widetilde{F}_s / \widetilde{F})^l .
\end{equation}
We fix such isomorphism \eqref{eq:8.8}, and hence \eqref{eq:8.9} as well. Another
choice would correspond to another separable closure of $F$, so that is harmless when
it comes to Langlands parameters. Take $r < l$ and recall the map 
$\mathbf{W}_F / \text{Gal}(F_s / F)^l \to F^\times / 1 + \mf p_F^r$ from local class 
field theory. By \cite[Proposition 3.6.1]{Del} the following diagram commutes:
\begin{equation}\label{eq:8.1}
\begin{array}{ccc}
F^\times / 1 + \mf p_F^r & \xrightarrow{\zeta_r} & 
\widetilde{F}^\times / 1 + \mf p_{\widetilde F}^r \\
\uparrow & & \uparrow \\
\mathbf{W}_F / \text{Gal}(F_s / F)^l & \xrightarrow{} & 
\mathbf{W}_{\widetilde F} / \text{Gal}(\widetilde{F}_s / \widetilde{F})^l
\end{array} .
\end{equation}
Notice that $G$ and $\widetilde G$ have the same Langlands dual group, namely 
$\GL_n (\C)$. Hence \eqref{eq:8.9} induces a bijection
\begin{equation}\label{eq:8.10}
\Phi_l^\zeta : \Phi_l (G) \to \Phi_l (\widetilde{G}) .
\end{equation}
In fact $\Phi_l^\zeta$ is already defined on the level of Langlands parameters 
without conjugation-equivalence, and in that sense $\Phi_l^\zeta (\phi)$ and $\phi$ 
always have the same image in $\GL_n (\C)$.
We remark that $\Phi_l^\zeta$ can be defined in the same way for $G^\sharp$ and 
$\widetilde{G}^\sharp$, because these groups have the common Langlands dual group 
$\PGL_n (\C)$.

We will prove that $\Phi_l^\zeta$ describes the effect that 
\[
\overline{\zeta_r^G} \colon \Irr (G,K_r) \to \Irr (\widetilde{G},\widetilde{K_r}) 
\]
has on Langlands parameters, when $l$ is large enough compared to $r \in \Z_{>0}$. 
First we do so for general linear groups over fields. The next result improves on
\cite[Corollary 8.8]{Gan2} and \cite[Theorem 2.3.11]{Gan1} in the sense that it gives 
an explicit lower bound on the $l$ for which the statement holds. 
However, our bound $l > 2^{n-1} r$ is still excessively large. We expect that 
the result is valid whenever $l > r$, but we did not manage to prove that.

\begin{thm}\label{thm:8.1}
Let $r \in \Z_{>0}$ and suppose that $F$ and $\widetilde F$ are $l$-close for some
$l > 2^{n-1} r$. Then the following diagram commutes:
\[
\begin{array}{ccc}
\Irr (\GL_n (F),K_r) & \xrightarrow{\overline{\zeta_r^{\GL_n (F)}}} & 
\Irr \big( \GL_n (\widetilde{F}), \widetilde{K}_r \big) \\
\downarrow \mathrm{rec}_{F,n} & & \downarrow \mathrm{rec}_{\widetilde{F},n} \\
\Phi_l (\GL_n (F)) & \xrightarrow{\Phi^\zeta_l} & \Phi_l \big( \GL_n (\widetilde {F}) \big) 
\end{array}
\]
\end{thm}
\begin{proof}
The proof will be conducted with induction to $n$. For $n=1$ the diagram becomes
\begin{equation}\label{eq:8.2}
\begin{array}{ccc}
\Irr (F^\times / 1 + \mf p_F^r) & \xrightarrow{\overline{\zeta_r^{F^\times}}} & 
\Irr \big( \widetilde{F}^\times / 1 + \mf p_{\widetilde F}^r \big) \\
\downarrow \mathrm{rec}_{F} & & \downarrow \mathrm{rec}_{\widetilde F} \\
\Irr \big( \mathbf{W}_F / \text{Gal}(F_s / F)^l \big) & \xrightarrow{\Phi^\zeta_l} & 
\Irr \big( \mathbf{W}_{\widetilde F} / \text{Gal}(\widetilde{F}_s / \widetilde{F})^l \big)  
\end{array} ,
\end{equation}
which commutes by Deligne's result \eqref{eq:8.1}. 

Now we fix $n>1$ and we assume the theorem for all $n' < n$. 
Consider a supercuspidal $\pi \in \Irr (\GL_n (F),K_r)$ with Langlands parameter 
$\phi = \mathrm{rec}_{F,n}(\pi) \in \Phi_l (\GL_n (F))$. By the construction of the local
Langlands correspondence for general linear groups, $\SL_2 (\C) \subset \ker \phi$ and
$\phi$ is elliptic. 
By Theorem \ref{thm:A.1} $\overline{\zeta_r^{\GL_n (F)}}(\pi) \in 
\Irr \big( \GL_n (\widetilde{F}), \widetilde{K}_r \big)$ is also supercuspidal and its 
central character is related to that of $\pi$ via \eqref{eq:A.7}.

Let $\widetilde{\phi}_l \in \Phi_l (\GL_n (\widetilde{F}))$ be its Langlands parameter and write 
$\phi_l = (\Phi^\zeta_l )^{-1}(\widetilde{\phi}_l)$. Clearly $\SL_2 (\C) \subset \ker \phi_l$ and 
$\phi_l$ is elliptic, so $\mathrm{rec}_{F,n}^{-1}(\phi_l)$ is supercuspidal. 
The commutative diagram \eqref{eq:8.2} says that $\mathrm{rec}_{F,n}^{-1}(\phi_l)$ has the
same central character as $\pi$. By Theorem \ref{thm:A.1}.e
\[
\vareps(s,\pi,\psi) = \vareps\big( s,\overline{\zeta_r^{\GL_n (F)}}(\pi), \widetilde \psi \big) 
= \vareps \big( s,\widetilde{\phi}_l, \widetilde \psi \big) .
\]
By \cite[Proposition 3.7.1]{Del} the right hand side equals
\[
\vareps \big( s,\widetilde{\phi}_l, \widetilde \psi \big) = \vareps (s,\phi_l,\psi) = 
\vareps (s,\mathrm{rec}_{F,n}^{-1}(\phi_l),\psi) ,
\]
so $\mathrm{rec}_{F,n}^{-1}(\phi_l)$ has the same $\vareps$-factor as $\pi$. Now we consider 
any generic \\ 
$\pi' \in \Irr (\GL_{n-1}(F),K_{2r,n-1})$ with Langlands parameter $\phi'$. The induction 
hypothesis and Theorem \ref{thm:A.3} apply to $\pi'$ because $2^{n-2} 2r < l$. 
By Theorem \ref{thm:A.3}, the induction hypothesis and \cite[Proposition 3.7.1]{Del}:
\begin{equation}
\begin{split}
\vareps (s,\pi \times \pi',\psi) & = \vareps \big( s, \overline{\zeta_{2r}^{\GL_n (F)}}(\pi) 
\times \overline{\zeta_{2r}^{\GL_{n-1}(F)}} (\pi'), \widetilde{\psi} \big) \\
& = \vareps \big( s, \overline{\zeta_r^{\GL_n (F)}}(\pi) \times 
\mathrm{rec}_{\widetilde{F},n-1}^{-1} (\Phi^\zeta_l (\phi')), \widetilde{\psi} \big) \\
& = \vareps \big( s,\widetilde{\phi}_l \otimes \Phi^\zeta_l (\phi'), \widetilde{\psi} \big) \\
& = \vareps (s,\phi_l \otimes \phi',\psi) \; = \; 
\vareps (s,\mathrm{rec}_{F,n}^{-1}(\phi_l) \times \pi',\psi) .
\end{split}
\end{equation}
Together with Theorem \ref{thm:LLCr} this implies $\pi \cong \mathrm{rec}_{F,n}^{-1}(\phi_l)$. 
Hence the diagram of the theorem commutes for supercuspidal $\pi \in \Irr (\GL_n (F),K_r)$.

For non-supercuspidal representations in $\Irr (\GL_n (F),K_r)$ it is easier. As already
discussed in Section \ref{sec:GLn}, the extension of the LLC from supercuspidal 
representations to $\Irr (\GL_n (F))$ is based on the Zelevinsky classification \cite{Zel}.
More precisely, the LLC is determined by:
\begin{itemize}
\item the parameters of supercuspidal representations;
\item the parameter of the Steinberg representation;
\item compatibility with unramified twists; 
\item compatibility with parabolic induction followed by forming Langlands quotients. 
\end{itemize}
The Steinberg representation $\St$ of $\GL_n (F)$ is the only irreducible essentially square-integrable
in the unramified principal series, which is tempered and has a real infinitesimal central
character. By Theorem \ref{thm:A.1} the functor $\overline{\zeta_r^{\GL_n (F)}}$ preserves all
these properties, so it matches the Steinberg representations of $\GL_n (F)$ and 
$\GL_n (\widetilde{F})$. The Langlands parameter of $\St$ is trivial on $\mathbf{W}_F$ and
its restriction to $\SL_2 (\C)$ is the unique irreducible $n$-dimensional representation of that
group. This holds over any local non-archimedean field, so $\Phi_l^\zeta$ matches
the Langlands parameters of the Steinberg representations of $\GL_n (F)$ and $\GL_n (\widetilde{F})$.
By Theorem \ref{thm:A.1} the functor $\overline{\zeta_r^{\GL_n (F)}}$ and its versions for groups
of lower rank respect unramified twists, parabolic induction and Langlands quotients.

To determine the Langlands parameters of elements of $\Irr (\GL_n (F),K_r)$ via the above method,
one needs only representations (possibly of groups of lower rank) that have nonzero 
$K_r$-invariant vectors. We checked that in every step of this method the effect of
$\overline{\zeta_r^{\GL_n (F)}}$ on the Langlands parameters is given by $\Phi^\zeta_l$. Hence
the diagram of the theorem commutes for all representations in $\Irr (\GL_n (F),K_r)$.
\end{proof}

Because the LLC for inner forms of $\GL_n (F)$ is closely related to that for $\GL_n (F)$ itself,
we can generalize Theorem \ref{thm:8.1} to inner forms.

\begin{thm}\label{thm:8.2}
Let $G = \GL_m (D)$ and $\widetilde{G} = \GL_m (\widetilde{D})$, with the same Hasse invariant.
For any $r \in \N$ there exists $l > r$ such that, whenever $F$ and $\widetilde{F}$ are $l$-close,
the following diagram commutes:
\[
\begin{array}{ccc}
\Irr (G,K_r) & \xrightarrow{\overline{\zeta_r^G}} & 
\Irr (\widetilde{G}, \widetilde{K}_r) \\
\downarrow \mathrm{rec}_{D,m} & & \downarrow \mathrm{rec}_{\widetilde{D},m} \\
\Phi_l (G) & \xrightarrow{\Phi^\zeta_l} & \Phi_l (\widetilde {G}) 
\end{array}
\]
In other words, Theorem \ref{thm:8.1} also holds for inner forms of $\GL_n (F)$, but without
an explicit lower bound for $l$.
\end{thm}
\begin{proof}
The bijection \eqref{eq:5.2} shows that we can write any $\pi \in \Irr (G,K_r)$ as the
Langlands quotient $L(P,\omega)$ of $I_P^G (\omega)$, where $P$ is a standard parabolic subgroup,
$M$ is Levi factor of $P$ and $\omega \in \Irr_{\ess L^2}(M)$. Moreover we may assume that
$M = \prod_j \GL_{m_j}(D)$ and $\omega = \otimes_j \omega_j$. The fact that $\pi$ has nonzero 
$K_r$-invariant vectors implies $\omega_j \in \Irr (\GL_{m_j}(D),K_r)$. 
By construction \eqref{eq:5.5}
\begin{equation}\label{eq:8.4}
\mathrm{rec}_{D,m}(\pi) = \prod\nolimits_j \mathrm{rec}_{D,m_j}(\omega_j) = 
\prod\nolimits_j \mathrm{rec}_{F,d m_j} (\JL (\omega_j)) .
\end{equation}
The right hand side forces us to compare the Jacquet--Langlands correspondence with the method
of close fields. In fact, this is how Badulescu proved this correspondence over local fields
of positive characteristic. It follows from \cite[p. 742--744]{Bad1} that there exist $l > r' \geq r$
such that, whenever $F$ and $\widetilde{F}$ are $l$-close, the following diagram commutes
for all $k \leq m$:
\begin{equation}\label{eq:8.5}
\begin{array}{ccc}
\Irr_{\ess L^2}(\GL_k (D),K_r) & \xrightarrow{\overline{\zeta_r^{\GL_k (D)}}} & 
\Irr (\GL_k (\widetilde{D}), \widetilde{K}_r) \\
\downarrow \JL & & \downarrow \JL \\
\Irr_{\ess L^2}(\GL_{kd} (F),K_{r'}) & \xrightarrow{\overline{\zeta_r^{\GL_{kd} (F)}}} & 
\Irr (\GL_{kd} (\widetilde{F}), \widetilde{K}_{r'}) 
\end{array}
\end{equation}
Enlarge $l$ so that Theorem \ref{thm:8.1} applies to $\Irr (\GL_{kd}(F),K_{r'})$ for all $k \leq m$.
By Theorem \ref{thm:A.1}.c
\[
\overline{\zeta_r^G}(\pi) = L \big( \widetilde{P},\overline{\zeta_r^M}(\omega) \big) = 
L \big( \widetilde{P},\otimes_j \overline{\zeta_r^{\GL_{m_j}(D)}}(\omega_j) \big) . 
\]
Now \eqref{eq:8.5} shows that 
\begin{equation}\label{eq:8.6}
\JL \big( \overline{\zeta_r^M}(\omega) \big) = 
\otimes_j \JL \big( \overline{\zeta_r^{\GL_{m_j}(D)}}(\omega_j) \big) =
\otimes_j \overline{\zeta_{r'}^{\GL_{d m_j}(F)}}(\JL (\omega_j)) .
\end{equation}
By \eqref{eq:5.5} and Theorem \ref{thm:8.1} 
\[
\mathrm{rec}_{\widetilde{D},m} \big( \overline{\zeta_r^G}(\pi) \big) = \prod\nolimits_j 
\mathrm{rec}_{\widetilde{F},d m_j} \big( \overline{\zeta_{r'}^{\GL_{d m_j}(F)}}(\JL (\omega_j)) \big)
= \prod\nolimits_j \Phi^\zeta_l \big( \mathrm{rec}_{F,d m_j} (\JL (\omega_j)) \big) . 
\]
Comparing this with \eqref{eq:8.4} concludes the proof.
\end{proof}

Now we are ready to complete the proof of Theorem \ref{thm:6.3}, and hence of our main
result Theorem \ref{thm:6.5}.\\

\noindent
\emph{Proof of Theorem \ref{thm:6.3} when $\char(F) = p > 0$.}\\
Choose $r \in \N$ such that $\Pi_\phi (G) \in \Irr (G,K_r)$ and choose $l \in \N$ such
that Theorem \ref{thm:8.2} applies. Find a $p$-adic field $\widetilde{F}$ which is
$l$-close to $F$, fix a representative for $\phi$ and define $\widetilde{\phi}$ as the
map $\mathbf{W}_F \times \SL_2 (\C) \to \GL_n (\C)$ obtained from $\phi$ via \eqref{eq:8.9}.
Thus $\widetilde{\phi}$ is a particular representative for $\Phi_l^\zeta (\phi) \in \Phi_l (G)$.
By Theorem \ref{thm:8.2} $\Pi_{\widetilde \phi}(\widetilde G) = \overline{\zeta_r^G} 
(\Pi_\phi (G))$ and by Theorem \ref{thm:A.1}
\[
\mathrm{End}_{\widetilde G} \big( \Pi_{\widetilde \phi}(\widetilde G) \big) \cong
\mathrm{End}_G \big( \Pi_\phi (G) \big) .
\]
Let $\phi^\sharp \in \Phi (G^\sharp)$ and 
$\widetilde{\phi}^\sharp \in \Phi (\widetilde{G}^\sharp)$
be the Langlands parameters obtained from $\phi$ and $\widetilde \phi$ via the quotient
map $\GL_n (\C) \to \PGL_n (\C)$. By construction $\phi^\sharp$ and $\widetilde{\phi}^\sharp$ 
have the same image in $\PGL_n (\C)$, so
\begin{equation}\label{eq:8.7}
\mc S_{\widetilde{\phi}^\sharp} = \mc S_{\phi^\sharp} \quad \text{and} \quad
\mc Z_{\widetilde{\phi}^\sharp} = \mc Z_{\phi^\sharp} .
\end{equation}
With \eqref{eq:6.4} this provides natural isomorphisms
\[
X^G (\Pi_\phi (G)) \cong \mc S_{\phi^\sharp} / \mc Z_{\phi^\sharp} =
\mc S_{\widetilde{\phi}^\sharp} / \mc Z_{\widetilde{\phi}^\sharp} \cong
X^{\widetilde G} (\Pi_{\widetilde \phi} (\widetilde{G})) .
\]
In view of \eqref{eq:8.2}, the composite isomorphism 
$X^{\widetilde G} (\Pi_{\widetilde \phi} (\widetilde{G})) \cong X^G (\Pi_\phi (G))$ comes from
$F^\times / 1 + \mf p_F^r \cong \widetilde{F}^\times / 1 + \mf p_{\widetilde F}^r$. For 
$\widetilde{\gamma} \in X^{\widetilde G} (\Pi_{\widetilde \phi} (\widetilde{G}))$, choose
\[
I_{\widetilde \gamma} \in \mathrm{Hom}_{\widetilde G} \big( \Pi_{\widetilde \phi}(\widetilde{G}),
\Pi_{\widetilde \phi}(\widetilde{G}) \otimes \widetilde{\gamma} \big)
\] 
as in \cite[\S 12]{HiSa}. Then Theorem \ref{thm:A.1} yields intertwining operators \\
$I_\gamma \in \mathrm{Hom}_{G} \big( \Pi_{\phi}(G),\Pi_{\phi}(G) \otimes \gamma \big)$. 
Consequently
\begin{equation}
\kappa_{\phi^\sharp} (\gamma,\gamma') = I_\gamma I_{\gamma'} I_{\gamma \gamma'}^{-1} =
I_{\widetilde \gamma} I_{\widetilde{\gamma}'} I_{\widetilde{\gamma} \widetilde{\gamma}'}^{-1}
= \kappa_{\widetilde{\phi}^\sharp} (\widetilde{\gamma},\widetilde{\gamma}') .
\end{equation}
Because we already proved Theorem \ref{thm:6.3} for $\widetilde F$, this gives 
\begin{equation}
\C [\mc S_{\phi^\sharp} / \mc Z_{\phi^\sharp} ,\kappa_{\phi^\sharp} ] = 
\C [\mc S_{\widetilde{\phi}^\sharp} / \mc Z_{\widetilde{\phi}^\sharp}, 
\kappa_{\widetilde{\phi}^\sharp}] \cong e_{\chi_{\widetilde G}} \C[\mc S_{\phi^\sharp}]
= e_{\chi_G} \C [\mc S_{\phi^\sharp}] .
\end{equation}
That the isomorphism $\C [\mc S_{\phi^\sharp} / \mc Z_{\phi^\sharp} ,\kappa_{\phi^\sharp} ]
\cong e_{\chi_G} \C [\mc S_{\phi^\sharp}]$ is of the required form and that it is unique up
to twists by characters of $\mc S_{\phi^\sharp} / \mc Z_{\phi^\sharp}$ follows from the
corresponding statements over $\widetilde F$ and \eqref{eq:8.7}. $\qquad \Box$

\end{document}